\documentclass{amsart}

\usepackage{amsmath, amsthm, amssymb, graphicx, amsrefs}

\copyrightinfo{2020}{Rados{\l}aw K. Wojciechowski}

\theoremstyle{plain}
\newtheorem{theorem}{Theorem}[section]
\newtheorem{lemma}[theorem]{Lemma}
\newtheorem{proposition}[theorem]{Proposition}
\newtheorem{corollary}[theorem]{Corollary}

\theoremstyle{definition}
\newtheorem{definition}[theorem]{Definition}
\newtheorem{example}[theorem]{Example}

\theoremstyle{remark}
\newtheorem{remark}[theorem]{Remark}

\numberwithin{equation}{section}

\newcommand{\R}{{\mathbb R}}
\newcommand{\N}{{\mathbb N}}
\newcommand{\Z}{{\mathbb Z}}

\newcommand{\Lp}{\mathcal{L}}
\newcommand{\lm}{\lambda}

\newcommand{\al}{\alpha}
\newcommand{\pt}{\partial_t}
\newcommand{\eps}{\epsilon}
\newcommand{\ka}{\kappa}
\newcommand{\lra}{\longrightarrow}
\newcommand{\F}{\mathcal{F}}
\newcommand{\D}{\mathcal{D}}
\newcommand{\Q}{\mathcal{Q}}
\newcommand{\vr}{\varrho}
\newcommand{\si}{\sigma}
\newcommand{\vph}{\varphi}
\newcommand{\kap}{\kappa}

\newcommand{\Deg}{{\mathrm{Deg}}}

\newcommand{\Hm}[1]{\leavevmode{\marginpar{\tiny%
$\hbox to 0mm{\hspace*{-0.5mm}$\leftarrow$\hss}%
\vcenter{\vrule depth 0.1mm height 0.1mm width \the\marginparwidth}%
\hbox to 0mm{\hss$\rightarrow$\hspace*{-0.5mm}}$\\\relax\raggedright
#1}}}

\providecommand{\eat}[1]{}

\begin{document}

\title[Stochastic completeness]{Stochastic completeness of graphs: bounded Laplacians, intrinsic metrics, volume growth and curvature}
\author{Rados{\l}aw K. Wojciechowski}
\address{Graduate Center of the City University of New York, 365 Fifth Avenue, New York, NY, 10016.
}
\address{York College of the City University of New York, 94-20 Guy R. Brewer Blvd., Jamaica, NY 11451.}
\email{rwojciechowski@gc.cuny.edu}
\date{\today}
\thanks{The author gratefully acknowledges financial support from PSC-CUNY Awards, jointly funded by the Professional Staff Congress and the City University of New York, and the Collaboration Grant for Mathematicians, funded by the Simons Foundation.}

\begin{abstract}
The goal of this article is to survey various results concerning stochastic completeness of graphs. In particular,
we present a variety of formulations of stochastic completeness and 
discuss how a discrepancy between uniqueness class and volume growth criteria 
in the continuous and discrete settings
was ultimately resolved via the use of intrinsic metrics. 
Along the way, we discuss some equivalent notions of boundedness in the sense of geometry
and of analysis.
We also discuss various curvature criteria
for stochastic completeness and discuss how weakly spherically symmetric graphs
establish the sharpness of results.
\end{abstract}

\maketitle
\tableofcontents

\section{Introduction}
The goal of this survey paper is to give an overview of results for the uniqueness of bounded solutions
of the heat equation with continuous time parameter, aka stochastic completeness, on infinite weighted graphs.
We first discuss some equivalent formulations of stochastic completeness and how
we have to go beyond the realm of bounded operators on graphs
in order for this property to be of interest.  
Furthermore, we discuss how the combinatorial graph distance is not an appropriate choice
of metric for the purpose of finding results analogous to those found in the setting of Riemannian manifolds.
This leads to the use of so-called intrinsic metrics which give an appropriate notion of
volume growth. Finally, we discuss how some recently developed
versions of curvature on graphs can be used to give conditions for stochastic completeness.

\subsection{Stochastic completeness, uniqueness class and volume growth}
In 1986, Alexander Grigor$\cprime$yan published an optimal volume growth condition
for the stochastic completeness of a geodesically complete Riemannian manifold. More specifically, if
a geodesically complete manifold $M$ satisfies
$$\int^\infty \frac{r}{\log V(r)}dr =\infty$$
where $V(r)$ denotes the volume of a ball of radius $r$ with an arbitrary center, then 
$M$ is stochastically complete \cite{Gri86}. In particular, we note
that any Riemannian manifold with volume growth satisfying $V(r) \leq Ce^{r^2}$ for $C>0$
will be stochastically complete. This result improved all other known volume growth
criteria at the time \cite{Gaf59, KLi} and is sharp in the sense that there exist stochastically incomplete manifolds
with volume growth of order $e^{r^{2+\eps}}$ for any $\eps>0$. In particular, the class of model manifolds
already provides such examples, see the survey article of Grigor$'$yan for this and 
many other results \cite{Gri99}.

Grigor$'$yan's volume growth criterion was proven via a uniqueness class result for solutions
of the heat equation. More specifically, if $u$ is a solution of the heat equation on $M \times (0,T)$
with initial condition $0$ and for all $r$ large $u$ satisfies
$$\int_0^T \int_{B_r} u^2(x,t)\  d\mu \ dt \leq e^{f(r)}$$
where $f$ is a monotone increasing function on $(0,\infty)$ which satisfies
$$\int^\infty \frac{r}{f(r)}dr=\infty,$$
then $u=0$ on $M \times (0,T)$.
It follows that bounded solutions of the heat equation are unique by taking the difference
of two solutions and letting $f(r)=\log(C^2TV(r))$ where $C$ is a bound on the solutions.
See \cite{Gaf59, KLi, Dav92, Tak89} for other techniques for proving volume
growth criteria for stochastic completeness in the manifold setting.

In the setting of graphs, an explicit study of geometric conditions for the 
uniqueness of bounded solutions of the heat equation
with continuous time parameter can be found in
work of J{\'o}zef Dodziuk and Varghese Mathai \cite{DM06} as well as that of Dodziuk \cite{Dod06}
and subsequently taken up in the author's Ph.D. thesis \cite{Woj08} and independently
in work of Andreas Weber 
\cite{Web10}. In particular,
Dodziuk/Mathai show that whenever the Laplacian on a graph with standard weights is a bounded
operator, then the graph is stochastically complete. Dodziuk then extends this result
to allow weights on edges. The technique used to establish these results is that 
of a minimum principle for the heat equation.

In the thesis \cite{Woj08} and follow-up paper \cite{Woj09}
we rather use an equivalent formulation of stochastic completeness in terms
of bounded $\lm$-harmonic functions to derive criteria for stochastic completeness which allow for
unbounded operators. Furthermore, we give a full characterization of stochastic completeness
in the case of trees which enjoy a certain symmetry. This characterization
already shows a disparity between the graph and manifold
settings in that there  exist stochastically incomplete graphs with factorial volume growth, that is, 
if $V(r)$ denotes the  
number of vertices within $r$ steps of a center vertex, then the tree is stochastically
incomplete and $V(r)$ grows factorially with $r$.
However, a more striking disparity appeared in a subsequent paper 
which introduced a class of graphs called anti-trees. 
These graphs can be stochastically incomplete and have polynomial volume growth \cite{Woj11}. 
In particular, there exist stochastically incomplete anti-trees with volume growth like 
$$V(r)\sim r^{3+\eps}$$
for any $\eps>0$. This provides a very strong contrast with the borderline for the manifold case
given by Grigor$'$yan's result.

The volume growth in these examples involves taking balls via the usual combinatorial graph metric, that is,
taking the least number of edges in a path connecting two vertices. This notion reflects only the global connectedness properties of the graph. However, it is natural to expect that a metric should
also reflect the local geometry of a graph, i.e., the valence or degree of vertices. Furthermore, if the
graph has weights on both edges and a measure on vertices, then the metric should 
interact with both the edge weights and the vertex measure.
For Riemannian manifolds, there exists the notion of an intrinsic metric which naturally arises
from the energy form as well as from
the geometry of the manifold. This notion of an intrinsic metric for the energy form was then extended to strongly 
local Dirichlet forms by Karl-Theodor Sturm \cite{Stu94}. Now, graphs which have both a weight on 
edges and a measure on vertices can be put into a one-to-one correspondence
with regular Dirichlet forms on discrete measure spaces as discussed in work of Matthias Keller and Daniel Lenz \cite{KL12}. 
However, the Dirichlet forms that arise from graphs are not strongly local. Thus, the notion
of an intrinsic metric from strongly local Dirichlet froms has to be extended to the non-local setting. 
This was done systematically in a paper by Rupert Frank, Daniel Lenz and Daniel Wingert \cite{FLW14}. 

The notion of intrinsic metrics for non-local Dirichlet forms was quickly put to use by the graph theory community. 
A first example of a concrete intrinsic metric for weighted graphs already appers in the 
Ph.D. thesis of Xueping Huang \cite{Hua11b} and can also be found in the work of Matthew Folz 
on heat kernel estimates around the same time \cite{Fol11}.
However, even given the tool of intrinsic metrics, there are still difficulties in proving an analogue
to Grigor$'$yan's criterion for graphs. In particular, Huang gives an example of a graph for which there
exists a non-zero bounded solution of the heat equation $u$ with zero initial condition which satisfies
$$\int_0^T \int_{B_r} u^2(x,t)\  d\mu \ dt \leq e^{f(r)}$$
for $f(r)=C r \log r$ for some constant $C$, see \cite{Hua11b, Hua12}. 
Hence, as $f$ in this case clearly satisfies $\int^\infty {r}/{f(r)} \ dr=\infty$, 
we see that even when using intrinsic metrics, a direct analogue
to Grigor$'$yan's proof is not possible for all graphs. 

A recent breakthrough in resolving this issue can be found in the work of Xueping Huang, Matthias Keller
and Marcel Schmidt \cite{HKS}. In this paper, the authors first prove a uniqueness class result
which is valid for a certain class of graphs called globally local. They can then
reduce the study of stochastic completeness of general graphs to that of globally local graphs using
the technique of refinements first found in \cite{HS14}. With these two results, they are able to establish
an exact analogue to the volume growth criterion of Grigor$'$yan which is valid for all graphs. That is, letting
$V_\varrho(r)$ denote the measure of a ball with respect to an intrinsic metric and letting 
$\log^{\#}(x)=\max\{\log(x), 1\}$ if
$$\int^\infty \frac{r}{\log^{\#} V_\varrho(r)} dr = \infty,$$
then the graph is stochastically complete. We note that taking the minimum with 1 
is only necessary to cover the case of when the entire graph has small measure;
the actual value of the constant 1 is not relevant.

Let us mention that the volume growth criterion for stochastic completeness of graphs involving
intrinsic metrics was first proven under 
some additional assumptions by Folz \cite{Fol14}. 
The proof technique of Folz, however, is different from that of Grigor$'$yan.
More specifically, Folz bypasses Grigor$'$yan's uniqueness class technique via a
probabilistic approach involving synchronizing the random walk on the graph with a random walk
on an associated quantum graph and then applying a generalization of Grigor$'$yan's result for manifolds
to strongly local Dirichlet forms found in work of Sturm \cite{Stu94}. A similar proof involving quantum graphs
but using analytic techniques can also be found in a paper by Huang \cite{Hua14b}.

We would also like to highlight earlier work focused on a volume growth criterion 
by Alexander Grigor$'$yan, Xueping Huang and Jun Masamune \cite{GHM12} 
using a technique from \cite{Dav92}. While this did not yield the optimal volume growth condition when using
intrinsic metrics, it did yield an optimal volume growth condition for the combinatorial graph metric in that 
$$V(r) \leq Cr^3$$
implies stochastic completeness where $V(r)$ is the volume defined with respect to the combinatorial graph metric.
Thus, we see that the anti-tree examples found in \cite{Woj11} are the smallest stochastically incomplete
graphs in the combinatorial graph distance.

\subsection{Curvature and stochastic completeness}
Let us now turn to curvature. For Riemannian manifolds, in a paper from 1974, 
Robert Azencott gave both a curvature criterion for stochastic completeness and the first examples of stochastically incomplete manifolds \cite{Az74}.
In Azencott's example, the curvature decays to negative infinity rapidly, thus it is natural to expect that
lower curvature bounds are necessary for stochastic completeness.
An optimal result in this direction involving Ricci curvature
was established by Nicholas Varopoulos \cite{Var83} and Pei Hsu
\cite{Hsu89}. It can be formulated as follows: let $M$ be a geodesically complete Riemannian manifold
and suppose that $\ka$ is a positive increasing continuous function on $(0,\infty)$ such that for all points
away from the cut locus on the sphere of radius $r$ we have $\mathrm{Ric}(x)\geq -C \ka^2(r)$ for all $r$ large
and $C>0$. 
If
$$\int^\infty \frac{1}{\ka(r)}dr=\infty,$$
then $M$ is stochastically complete. This improved the previously known results which gave that Ricci curvature
uniformly bounded from below implied stochastic completeness as proven by Shing-Tung Yau \cite{Yau78}, 
see also the work of Dodziuk \cite{Dod83}. However, due to the connection between
Ricci curvature and volume growth, this result is already implied by Grigor$'$yan's volume growth result.
There is also a number of comparison results for stochastic completeness involving curvature, see
\cite{Ichi82} or Section 15 in the survey of Grigor$'$yan \cite{Gri99}.

In recent years, there has been a tremendous interest in notions of curvature
on graphs. We focus here on two formulation. One definition of curvature 
originates in work of Dominique Bakry
and Michele \'{E}mery on hypercontractive semigroups \cite{BE85}. Thus, we refer to it as Bakry--\'{E}mery curvature.
A second formulation comes from the work of Yann Ollivier on Markov chains on metric spaces in \cite{Oll07, Oll09}.
This was later modified to give an infinitesimal version by Yong Lin, Linyuan Lu and Shing-Tung Yau \cite{LLY11}
and then extended to the case of possibly unbounded operators on graphs by Florentin M{\"u}nch and the author \cite{MW19}. In any case, we refer to this as Ollivier Ricci curvature.
For Bakry--\'{E}mery curvature, Bobo Hua and Yong Lin proved that a uniform lower bound implies stochastic completeness
in \cite{HL17}.
On the other hand, in \cite{MW19} we prove that for Ollivier Ricci curvature, if
$$\ka(r) \geq -C\log r$$
for $C>0$ and all large $r$ where $\ka(r)$ denotes the spherical curvature on a sphere
of radius $r$, then the graph is stochastically complete.
This is optimal in the sense that for any $\eps>0$ there exist stochastically incomplete
graphs with $\ka(r)$ decaying like $-(\log r)^{1+\eps}$. Thus, there is still a disparity in this condition for graphs and
for the Ricci curvature condition for manifolds as presented above. However, this disparity cannot
be resolved by using the notion of intrinsic metrics as was the case for stochastic completeness and volume growth.
 
\subsection{The structure of this paper}
We now briefly discuss the structure of this paper. Although we do not give full proofs of results,
we also do not assume any particular background of the reader and thus try to make the presentation
as self-contained as possible in terms of concepts and definitions. We also give specific references
for all results that we do not prove completely.

In Section~\ref{s:heat} we introduce the setting of weighted graphs and discuss the heat
equation. In particular, we outline an elementary construction of bounded solutions of the 
heat equation using exhaustion techniques. In Section~\ref{s:sc} we present some equivalent formulations for 
stochastic completeness. In particular, stochastic completeness is equivalent to the uniqueness of this bounded solution
of the heat equation. In Section~\ref{s:bounded} we discuss boundedness of the Laplacian and
how boundedness is related to stochastic completeness. In Section~\ref{s:symmetry} we then introduce
the class of weakly spherically symmetric graphs and present the examples of anti-trees which show the
disparity between the continuous and discrete settings in the case of the combinatorial graph metric.
In Section~\ref{s:intrinsic} we introduce intrinsic metrics and discuss how they can differ from the
combinatorial graph metric and how this is related to stochastic completeness. 
Finally, in Sections~\ref{s:volume}~and~\ref{s:curvature} we present the criteria
for stochastic completeness in terms of volume growth and curvature mentioned above.

We also mention here a recent survey article by Bobo Hua and Xueping Huang which has some contact
points with our article but also discusses heat kernel estimates, ancient solutions of the heat
equation and upper escape rates \cite{HH}.

\section{The heat equation on graphs}\label{s:heat}
\subsection{Weighted graphs}
We start by introducing our setting following \cite{KL12}. We note that this setting is very general
in that we allow for arbitrary weights on both edges and vertices. We also do not assume local finiteness,
i.e., that every vertex has only finitely many edges coming out of the vertex.

\begin{definition}[Weighted graphs] 
We let $X$ be a countably infinite set whose elements we refer to as 
\emph{vertices}. We then let $m:X \lra (0,\infty)$ denote a \emph{measure} on the vertex set which can
be extended to all subsets by additivity. Finally, we let $b:X \times X \lra [0,\infty)$ denote a function
called the \emph{edge weight} which satisfies
\begin{enumerate}
\item[(b1)] \quad $b(x,x)=0$ for all $x \in X$
\item[(b2)] \quad $b(x,y)=b(y,x)$ for all $x, y \in X$
\item[(b3)] \quad $\sum_{y \in X} b(x,y) <\infty$ for all $x \in X$.
\end{enumerate}
Whenever $b(x,y)>0$, we think of the vertices $x$ and $y$ as being \emph{connected} by an edge with 
weight $b(x,y)$, call $x$ and $y$ \emph{neighbors} and write $x \sim y$. 
Thus, (b1) gives that there are no loops,
(b2) that edge weights are symmetric and (b3) that the total sum of the edge weights is finite.
We call the triple $G=(X, b, m)$ a \emph{weighted graph} or just \emph{graph} for short.
\end{definition}

We note, in particular, that condition (b3) above allows for a vertex to have infinitely many neighbors.
Whenever, each vertex has only finitely neighbors, we call the graph \emph{locally finite}. We call the quantity
$$\Deg(x)= \frac{1}{m(x)}\sum_{y \in X} b(x,y)$$
the \emph{weighted vertex degree} of $x \in X$ or just \emph{degree} for short. We will see
that this function plays a significant role in what follows.

\begin{example}
We now present some standard choices for $b$ and $m$ to help orient the reader.
In particular, we discuss the case of standard edge weights, counting and degree measures.
\begin{enumerate}
\item Whenever $b(x,y) \in \{0,1\}$ for all $x, y \in X$, we say that the graph has \emph{standard edge weights}.
In this case, it is clear that condition (b3) in the definition of the edge weights 
implies that the graph must be locally finite.
\item One choice of vertex measure is the counting measure, that is, $m(x)=1$
for all $x \in X$. In this case, 
$m(K)=\# K$ is just the cardinality of any finite subset $K$. In the case of standard edge weights
and counting measure, we then
obtain
$$\Deg(x) =\# \{y \ | \ y \sim x\}$$
so that the weighted vertex degree is just the number of neighbors of $x$, that is, the valence
or degree of a vertex.
\item Another choice for the vertex measure is 
$$m(x) = \sum_{y \in X}b(x,y)$$
for $x \in X$. 
In the case of standard edge weights, 
it then follows that $m(x) = \# \{ y \ | \ y \sim x \}$ is the number of neighbors of $x$.
In any case, with this choice of measure, it is clear that
$$\Deg(x)=1$$ 
for all $x \in X$. 
\end{enumerate}
\end{example}

Often we will assume that graphs are \emph{connected} in the usual geometric sense, namely,
for any two vertices $x, y \in X$, there exists a sequence of vertices $(x_k)_{k=0}^n$ with $x_0=x$,
$x_n =y$ and $x_k \sim x_{k+1}$ for $k=0, 1, \ldots n-1$. We note that we include the case of $x=y$
when a vertex can be connected to itself via a path consisting of a single vertex and thus no edges.
Such a sequence is called a \emph{path}
connecting $x$ and $y$. We then let 
$$d:X \times X \lra [0,\infty)$$
denote the \emph{combinatorial graph distance} on $X$,
that is, $d(x,y)$ equals the least number of edges in a path connecting $x$ and $y$.
We note that this metric only considers the combinatorial properties of the graph encoded in $b$ but not
the actual value of $b(x,y)$ nor the vertex measure $m$. We will have more to say about this later.

\subsection{Laplacians and forms}
We now denote the set of all functions on $X$ by $C(X)$, that is,
$$C(X) = \{ f: X \lra \R\}$$
and the subset of finitely supported functions by $C_c(X)$.  The Hilbert space
that we will be interested in at various points is $\ell^2(X,m)$, the space of square summable functions
on $X$ with respect to the measure $m$. That is,
$$\ell^2(X,m) = \{ f \in C(X) \ | \ \sum_{x \in X}f^2(x)m(x)<\infty\}$$
with inner product
$\langle f, g \rangle = \sum_{x \in X}f(x)g(x)m(x)$. 

In order to introduce a formal Laplacian, we first have to restrict to a certain class of functions
as we do not assume local finiteness so that summability becomes an issue. 
\begin{definition}[Formal Laplacian and energy form]
We let
$$\F=\{ f\in C(X) \ | \ \sum_{y \in X} b(x,y) |f(y)| <\infty \textup{ for all } x \in X \}$$
and for $f \in \F$, we let 
$$\Lp f(x) = \frac{1}{m(x)} \sum_{y \in X} b(x,y) (f(x)-f(y))$$
for $x \in X$. The operator $\Lp$ is then called the \emph{formal Laplacian} associated to $G$. We furthermore let
$$\D=\{ f\in C(X) \ | \  \sum_{x, y \in X} b(x,y) (f(x)-f(y))^2 <\infty\}$$
denote the space of \emph{functions of finite energy}. For $f, g \in \D$, we let
$$\Q(f,g) = \frac{1}{2}\sum_{x,y \in X} b(x,y) (f(x)-f(y))(g(x)-g(y)) $$
denote the \emph{energy form} associated to $G$.
\end{definition}

We denote the restriction of $\Q$ to $C_c(X) \times C_c(X)$ by $Q_c$. It then follows that a 
version of Green's formula holds for $Q_c$:
$$Q_c(\vph,\psi) = \sum_{x \in X} \Lp \vph(x)\psi(x) m(x) = \sum_{x \in X} \vph(x) \Lp \psi(x) m(x) $$
for all $\vph, \psi \in C_c(X)\subseteq \ell^2(X,m)$, see, for example \cite{HK11}.
The form $Q_c$ is closable and thus there exists a 
unique self-adjoint operator $L$ with domain 
$D(L) \subseteq \ell^2(X,m)$ associated to the closure of $Q_c$ denoted by $Q$. 
For a discussion of the closure of a form and the construction of the 
associated operator in a general Hilbert space, see Theorem~5.37 in \cite{Weid80}.
We refer to $L$ as the \emph{Laplacian} associated to the graph $G$. We note that with our sign
convention, we have
$$\langle Lf, f \rangle = Q(f,f) \geq 0$$
for all $f \in D(L)$ so that $L$ is a positive operator.

\subsection{A word about essential self-adjointness and $\ell^2$ theory}
Although not a main concern of this article as we mostly deal with bounded solutions, 
we want to mention another approach to the construction of the Laplacian $L$.
In this viewpoint, one starts by restricting $\Lp$ to $C_c(X)$ and denoting the resulting operator by $L_c$,
that is, $D(L_c)=C_c(X)$ and $L_c$ acts as $\Lp$. 

However, due to the lack of local finiteness,
$\Lp$ does not necessarily map $C_c(X)$ into $\ell^2(X,m)$. Thus, whenever we want to consider $L_c$ as an operator on $\ell^2(X,m)$, we have to assume
that $\Lp$ maps $C_c(X)$ into $\ell^2(X,m)$.
Under this additional assumption, it is easy to see that $L_c$ is a symmetric operator on $\ell^2(X,m)$, i.e.,
$$\langle L_c \vph, \psi \rangle = \langle \vph, L_c \psi \rangle$$
for all $\vph, \psi \in C_c(X)$ and that the Green's formula reads as
$$Q_c(\vph,\psi) = \langle L_c \vph, \psi \rangle$$
for $\vph, \psi \in D(L_c)$. In this case, the self-adjoint operator associated to the closure of $Q_c$, which is just 
the Laplacian $L$,
is called the \emph{Friedrichs extension} of $L_c$, see
Theorem~5.38 in \cite{Weid80} for further details on the construction of this extension for general Hilbert
spaces. 

We note that $\Lp$ maps $C_c(X)$ into $\ell^2(X,m)$ whenever 
$\mathcal{L}1_x \in \ell^2(X,m)$ for all $x \in X$ where $1_x$
denotes the characteristic function of the singleton set $\{x\}$. It is a direct calculation that 
$\mathcal{L}1_x \in \ell^2(X,m)$ for all $x \in X$ if and only if 
$$\sum_{y \in X} \frac{b^2(x,y)}{m(y)} <\infty$$
for all $x \in X$. In particular, all locally finite graphs or, more generally, all graphs
with $\inf_{y \sim x} m(y)>0$ automatically satisfy this assumption.
Furthermore, the condition $\mathcal{L}1_x \in \ell^2(X,m)$ for all $x \in X$ is equivalent
to a variety of other conditions, for example, that $C_c(X) \subseteq D(L)$, for more details,
see \cite{Kel15}.

We further note that, in general, $L_c$ may have many self-adjoint extensions and
that processes associated to these different extensions may have different stochastic properties.
When $L_c$ has a unique self-adjoint extension, $L_c$ is called \emph{essentially self-adjoint}.
It was first shown as Theorem~1.3.1 in \cite{Woj08} that $L_c$ is essentially self-adjoint in the case 
of standard edge weights and counting measure. This was then extended to allow for general edge 
weights and any measure such that the measure of infinite paths is infinite as Theorem~6 in \cite{KL12}. 
This criterion was further improved and generalized in \cite{Gol14, Schm20} 
which consider more general operators on graphs.
For further discussion of essential self-adjointness, see \cite{HKLW12, HKMW13, Schm20} and reference therein.
We will also discuss the connection between essential self-adjointness and metric completeness in Subsection~\ref{ss:metric_esa} below.

\subsection{The heat equation: existence of solutions}
We now introduce a continuous time heat equation on $G$. We let $\ell^\infty(X)$ denote the set of bounded
functions on $X$, that is,
$$\ell^\infty(X) = \{ f \in C(X) \ | \ \sup_{x \in X} |f(x)| <\infty \}.$$
We now make precise the requirements
for summability, differentiability and boundedness of a solution.

\begin{definition}[Bounded solution of the heat equation]
Let $u_0 \in \ell^\infty(X).$
By a \emph{bounded solution
of the heat equation with initial condition} $u_0$ we mean a bounded function
$$u:X \times [0,\infty) \lra \R$$
such that $u(x,\cdot)$ is continuous for every $t\geq0$, differentiable
for $t>0$ and all $x \in X$ and
$$(\mathcal{L} + \pt)u(x,t) = 0$$
for all $x \in X$ and $t>0$ with $u(x,0)=u_0(x)$.
\end{definition}

We note, in particular, that as $u(\cdot, t) \in \ell^\infty(X)$ for every $t \geq 0$, we obtain
that $u(\cdot, t) \in \F$. Thus, we may apply the formal Laplacian to $u$ at every time $t \geq 0$.

The definition of bounded solutions raises two immediate questions: the existence and uniqueness of solutions. 
We will first address existence by showing that there always
exists a bounded solution which is minimal in a certain sense. On the other hand, uniqueness 
is one of the various formulations of stochastic completeness as we will discuss in the next section.

We note that although we have a self-adjoint operator $L$ on $\ell^2(X,m)$ so that we may apply the spectral
theorem and 
functional calculus to obtain a heat semigroup $e^{-tL}$ for $t \geq 0$, this semigroup acts on $\ell^2(X,m)$ and we
are actually interested in bounded solutions, i.e., solutions in $\ell^\infty(X)$. There is a number
of ways around this. One approach taken in \cite{KL12} is to extend the heat semigroup on $\ell^2(X,m)$
to all $\ell^p(X,m)$ spaces for $p \in [1,\infty]$ via monotone limits. Another approach is via the general
theory of Dirichlet forms and interpolation between $\ell^p(X,m)$ spaces, see \cite{Dav90, FOT94}.
We highlight a slightly different approach
in that we rather exhaust the graph via finite subgraphs, apply the spectral theorem to each operator
on the finite subgraph in order to get a solution and then take the limit. This rather elementary approach
has its roots in \cite{Dod83} which gave the first construction of the heat kernel
on a general Riemannian manifold without any geodesic completeness assumptions.

A basic tool behind the construction is the following minimum principle.
We call a subset $K \subseteq X$ \emph{connected} if any two vertices in $K$ can be connected via a
path that remains within $K$.
\begin{lemma}[Minimum principle for the heat equation]\label{l:minimum_principle}
Let $G$ be a connected weighted graph and let $K \subset X$ be a finite connected subset.  
Let $ T\ge0 $ and let $ u:X\times[0,T] \longrightarrow \R $ be such that $t \mapsto u(x,t) $ is continuously 
differentiable on $(0,T)$ for every $ x\in K $ and $u(\cdot, t) \in \mathcal{F}$ for all $t\in(0,T]$. 
Assume that $u$ satisfies
\begin{itemize}
\item[(A1)] \quad $ (\mathcal{L}+\pt) u\ge0 $ on $ K\times (0, T) $
\item[(A2)] \quad $ u\ge0 $ on $ \left(X\setminus K \times (0,T] \right) \cup (K\times \{0\}). $
\end{itemize}
Then, $ u\ge0 $ on $ K\times [0,T] $.
\end{lemma}
\begin{proof}
Suppose to the contrary that there exists $(x_0,t_0) \in K \times [0,T]$ such that $u(x_0,t_0) <0$.
By continuity, we can assume that $(x_0,t_0)$ is a minimum for $u$ on $K \times [0,T]$. 
By assumption (A2), it follows
that $t_0 >0$ so that $\pt u(x_0,t_0) \leq 0$. Furthermore, by the definition of $\Lp$, at a minimum we have
$\Lp u(x_0,t_0) \leq 0$. Therefore, $(\Lp + \pt)u(x_0,t_0) \leq 0$ and assumption (A1) 
gives $(\mathcal{L}+\pt) u(x_0,t_0)=0 $ from
which 
$$\Lp u(x_0,t_0)= \frac{1}{m(x_0)}\sum_{y \in X} b(x_0,y)(u(x_0,t_0)-u(y,t_0))=0$$ 
follows. Therefore, since we are at a minimum, we now obtain that
$$u(y,t_0)=u(x_0,t_0)<0$$
for all $y \sim x_0$. Iterating this argument and using the connectedness of $K$ now gives a contradiction
to (A2) as $K \neq X$ and we assume that $G$ is connected.
\end{proof}

\begin{remark}
We note that the finiteness of $K$ is not necessary. It suffices to assume that there is at least
one vertex outside of $K$ and that the negative part of $u$
attains a minimum on $K \times [0,T]$. The minimum principle then follows with basically the same proof,
see, for example, Lemma~3.5 in \cite{KLW13}. However, assuming the finiteness of $K$ is sufficient
for our purposes. For a much more elaborate discrete integrated minimum principle for solutions of the heat 
equation, see Lemma~1.1 in \cite{Hua12}.
\end{remark}

We now sketch the construction of the minimal bounded solution of the heat equation. 
We note that if $G$ is not connected, we work on each connected component of $G$ separately.
Thus, for the construction, we can assume without loss of generality that $G$ is connected.
We let $(K_n)_{n=0}^\infty$
be an \emph{exhaustion sequence} of the graph $G$ by which we mean that each $K_n$ is finite and connected,
$K_n \subseteq K_{n+1}$ and $X=\bigcup_n K_n$. For each $n$, we let $L_n$ denote the restriction of $\Lp$ to
$C(K_n)=\ell^2(K_n,m)$. More precisely, for a function $f \in C(K_n)$, we extend $f$ by 0 to be defined on all of $X$ and let
$L_n f(x) = \Lp f(x)$ for $x \in K_n$. Then, $L_n$ is an operator on a finite dimensional Hilbert space and we can define
$$e^{-tL_n}=\sum_{k=0}^\infty \frac{(-t)^k}{k!}L_n^k$$
for $t \geq 0$. We then define the \emph{restricted heat kernels} $p_t^n(x,y)$ for $t \geq 0$ and $x, y \in K_n$ 
via
$$p_t^n(x,y)=e^{-tL_n}\hat{1}_y(x)$$ 
where $\hat{1}_y = 1_y/m(y)$.
It is immediate that  
$$u_n(x,t)=e^{-tL_n}u_0(x)=\sum_{y\in K_n}p_t^n(x,y)u_0(y)m(y)$$
satisfies the heat equation
on $K_n \times [0,\infty)$ with initial condition $u_0$. 

Furthermore, applying Lemma~\ref{l:minimum_principle}, gives $0 \leq p_t^n(x,y)m(y) \leq 1$ and
$p_t^n(x,y) \leq p_t^{n+1}(x,y)$ for all $x,y \in K_n, t \geq 0$ and $n \in \N$. 
Thus, we may take the limit
$$p_t^n(x,y) \to p_t(x,y)$$ 
as $n \to \infty$  to define $p_t(x,y)$ which is called the 
\emph{heat kernel} on $G$. 
Then, by applying Dini's theorem and monotone convergence, we can show that
$$u(x,t)=\sum_{y \in X}p_t(x,y)u_0(y)m(y)$$ 
is a bounded solution of the heat equation with initial condition $u_0$ on $G$. 
For further details and proofs, see Section~2 in \cite{Woj08} for the case of standard edge weights
and counting measure. An alternative approach for general graphs involving resolvents is given
in Section 2 of \cite{KL12}, in particular, Proposition 2.7. 

\begin{remark}
The approach via resolvents in \cite{KL12}
is equivalent to the heat kernel approach above via the Laplace transform formulas, that is,
$$e^{-tL_n}= \lim_{k\to \infty}\left(\frac{k}{t} \left(L_n+\frac{k}{t}\right)^{-1}\right)^k$$
for all $t>0$ and
$$(L_n + \al)^{-1} = \int_0^\infty e^{-t\al}{e^{-tL_n}} dt$$
for all $\al>0$. Both of these formulas also hold for the Laplacian $L$ defined 
on the entire $\ell^2(X,m)$ space.
We further note that $L_n$ is a positive definite operator on $\ell^2(K_n,m)$
as can be seen by direct calculation which gives
$$\langle L_n f, f \rangle = \frac{1}{2}\sum_{x, y \in K_n} b(x,y)(f(x)-f(y))^2 
+ \sum_{x \in K_n}f^2(x) \sum_{y \not \in K_n} b(x,y)$$
for all $f \in \ell^2(K_n,m)$.
\end{remark}

We mention two further properties that follow from the construction and Lemma~\ref{l:minimum_principle}
above. First, the solution $u$ is minimal in the following sense: if $u_0\geq0$ and $w \geq 0$ is any
solution of the heat equation with initial condition $u_0$, then $u \leq w$. Secondly, as $0 \leq \sum_{y \in K_n} p_t^n(x,y)m(y) \leq 1$ for all $n$, we
get
$$0 \leq \sum_{y \in X} p_t(x,y)m(y) \leq 1$$
by taking the limit $n \to \infty$. 
We will return to the second inequality in the following section.

We note that the approach above also gives that
if $f \in \ell^\infty(X)$ with $0 \leq f \leq 1$, then
$$0 \leq \sum_{y \in X} p_t(x,y)f(y)m(y) \leq 1.$$ 
This property is referred to by saying that the heat semigroup is \emph{Markov}. The fact
that the semigroup is Markov will
be used later in our discussion of curvature on graphs.

\section{Formulations of stochastic completeness}\label{s:sc}
\subsection{Stochastic completeness and uniqueness of solutions}
We have seen that given any bounded function, we can construct a bounded
solution of the heat equation with the given function as an initial condition. 
We now address the uniqueness of this solution.
In fact, we will see that the uniqueness is equivalent to the following property.

\begin{definition}[Stochastic completness]
Let $G$ be a weighted graph. If for all $x \in X$ and all $t \geq 0$, 
$$\sum_{y \in X} p_t(x,y)m(y)=1,$$ then 
$G$ is called \emph{stochastically complete.} Otherwise, $G$ is called \emph{stochastically incomplete}.
\end{definition}

\begin{remark}
There is a short way to state the definition above which we will have recourse to later in our
discussion of curvature. Namely, letting $P_t = e^{-tL}$ denote the heat semigroup
on $\ell^2(X,m)$ for $t \geq 0$ defined via the spectral theorem, it follows that $P_t$ can be extended
to $\ell^\infty(X)$, the space of bounded functions, via monotone limits, see Section~6 in \cite{KL12} where
this is actually shown for all $\ell^p(X,m)$ spaces with $p \in [1,\infty]$. In particular, letting $1 \in \ell^\infty(X)$
denote the function which is constantly 1 on all vertices, we then have
$$P_t 1(x) = \sum_{y \in X}p_t(x,y)m(y)$$
for $x \in X$. Thus, stochastic completeness can also be written as $P_t 1 =1$ for all $t \geq 0$.
\end{remark}

The goal of this section is to give a variety of characterizations for this property. We do not aim
to be exhaustive but rather highlight the characterizations that will be useful later in our presentation.
If $v \in \F$ satisfies $\Lp v = \lm v$ for $\lm \in \R$, then $v$ is called a $\lm$-\emph{harmonic} function.
In particular, the theorem below characterizes stochastic completeness in terms of non-existence
of $\lm$-harmonic bounded functions for $\lm<0$. This will be used in several places in what follows.

\begin{theorem}[Characterizations of stochastic completeness]\label{t:sc_characterizations}
Let $G$ be a weighted graph. The following statements are equivalent:
\begin{itemize}
\item[(i)] $G$ is stochastically complete.
\item[(ii)] Bounded solutions of the heat equations are uniquely determined by initial conditions.
\item[(ii$'$)] The only bounded solution of the heat equation with initial condition $u_0=0$ is $u=0$.
\item[(iii)] The only bounded solution to $\Lp v = \lm v$ for some/all $\lm <0$
is $v=0$.
\item[(iii$'$)] The only non-negative bounded solution to $\Lp v \leq \lm v$ 
for some/all $\lm<0$ is $v=0$
\item[(iv)] Every bounded function $v$ with $v^* = \sup v>0$ satisfies $\sup_{\Omega_\alpha} \Lp v \geq 0$
for every $\alpha < v^*$ where $\Omega_\alpha = \{ x\in X \ | \ v(x) > v^*- \alpha\}$.
\end{itemize}
\end{theorem}
\begin{remark}[History and intuition]
We give a partial history with references for the equivalences above in various settings. The equivalence of (i),
(ii), and (iii) for diffusion processes on Euclidean spaces goes back to \cite{Fel54, Has60}.
For manifolds, see Theorem~6.2~and~Corollary~6.3 in \cite{Gri99} which also gives further
historical references. For Markov processes on discrete spaces, these equivalence go back to \cite{Fel57, Reu57}. 
For a proof in the case of graphs with
standard edge weights and counting measure, see Theorem~3.1.3 in \cite{Woj08}. For an extension
to weighted graphs see Theorem~1~in~\cite{KL12} which deals also with a more general phenomenon
called stochastic completeness at infinity. This allows for a discussion of these properties
in the case of operators of the type Laplacian plus a positive potential, see also \cite{MS} for a 
discussion of this property in the case of manifolds.
Condition (iv) is referred to as a weak Omori-Yau maximum principle after the original
work in \cite{Om67, Yau75}. The equivalence
of (iv) and stochastic completeness was shown for manifolds in \cite{PRS03} and for graphs as
Theorem 2.2 in \cite{Hua11a}.

We would also like to mention some of the intuition behind the equivalences. Roughly speaking, as mentioned
in the introduction and discussed further below,
a large volume growth or curvature decay is required for
stochastic completeness to fail. Let us discuss how this large volume growth can cause the failure of the
other properties listed in the theorem above. First, failure of (i) means that
the total probability of the process determined by the Laplacian to remain in the graph when
starting at a vertex $x$ is less than 1 at some time.
Hence, under a large enough volume growth (or curvature decay) the process can be swept off the graph
to infinity in a finite time. 
Second, failure of (ii) means that there exists a non-zero bounded solution of the heat equation with zero
initial condition. In other words, a large volume growth can create something out of nothing.
Third, by looking at the equation $\Lp v = \lm v$ for $v>0$ and $\lm<0$,
we see that $v$ must increase at some neighbor of each vertex. Hence, stochastic incompleteness
means that there is a sufficient amount of space in the graph to accommodate this growth while keeping $v$ bounded.
This gives the intuition for the failure of condition (iii).
Finally, $\Lp v \leq -C< 0$ on a set of vertices where $v$ is near its supremum means that there is always more 
room for $v$ to grow in the graph. This gives the intuition behind the failure of (iv). 

To summarize, in order for any of the four conditions above to fail requires a large amount of
space in the graph. Conversely, if there is no large growth, then the process remains in the graph
and the graph cannot accommodate non-zero bounded solutions to various equations.
Thus, stochastic completeness is also referred to as \emph{conservativeness}
or \emph{non-explosion}.
\end{remark}

\begin{proof}[Sketch of the proof of Theorem~\ref{t:sc_characterizations}] 
We now sketch a proof. For full details, please see the references
given in the remark directly above.
To show the equivalence between (i) and (ii), observe that both the constant function 
1 and $u(x,t) = \sum_{y \in X}p_t(x,y)m(y)$ are bounded solutions of the heat equation with initial condition 1.
The equivalence between (ii) and (ii$'$) is shown by taking the difference of two solutions of the heat equation
with initial conditions $u_0$. 
The equivalence between (ii$'$) and (iii) can be established via the fact that if $u$ is a
bounded solution of the heat equation with initial condition 0, then
$v(x) = \int_0^\infty e^{t\lm}u(x,t) dt$ is a bounded $\lm$-harmonic function for $\lm<0$.
To show the equivalence between (iii) and (iii$'$) one can use exhaustion and minimum principle arguments.
Finally, if (iii) fails and $v$ is a non-trivial bounded $\lm$-harmonic function for $\lm<0$, then
letting $\alpha = \sup v/2$, it can be shown that $\mathcal{L} v \leq -C < 0$ on $\Omega_\alpha$ so that (iv)
fails.
Conversely, if (iv) fails, then there exists a bounded function $v$ such that $\mathcal{L}v \leq -C$ for $C>0$
on $\Omega_\alpha$ for some $0< \alpha<v^*$. 
Then, $w= (v+\alpha -v^*)_+$ is a non-negative non-trivial bounded function with
$\mathcal{L} w \leq \lambda w$ for $\lambda = -C/\alpha$. Thus, (iii$'$) fails.
\end{proof}

\begin{remark}
We note that connectedness of the graph and the
semigroup property can be used to show that if $\sum_{y \in X}p_t(x,y)m(y)=1$
holds for one $x \in X$ and one $t>0$, then it holds for all $x \in X$ and all $t > 0$. However, we do not
require connectedness for the equivalence of the properties above as later we will need to consider
a possibly unconnected scenario. Thus, in the definition, we assume that the sum is 1
for all $x \in X$ and all $t \geq 0$.
\end{remark}

\subsection{A Khas$'$minskii criterion}
We will also need another property which implies stochastic completeness. This is sometimes referred to 
as a Khas$'$minskii-type criterion after \cite{Has60}. The formulation below is Theorem~3.3~in~\cite{Hua11a}, the
proof given there uses the weak Omori-Yau maximum principle, that is, condition (iv) in 
Theorem~\ref{t:sc_characterizations}.
\begin{theorem}\label{t:Khas}
Let $G$ be a weighted graph. If there exists $v \in \F$ which satisfies $v \geq 0$, $v(x_n) \to \infty$
for all sequences of vertices with $\Deg(x_n) \to \infty$ and
$$\Lp v + f(v) \geq 0$$
on $X \setminus K$ where $K\subseteq X$ is a set such that  $\Deg$ is a bounded function on $K$ and 
$f:[0,\infty) \lra (0,\infty)$ is an increasing continuously differentiable function with 
$$\int^\infty \frac{1}{f(r)}dr = \infty,$$
then $G$ is stochastically complete.
\end{theorem}

\begin{remark}
This formulation of a Khas$'$minskii-type criterion is very precise
as it involves the weighted degree function
as well as the function $f$. 
A more general formulation is that the existence of
a function $v$ which satisfies $\Lp v \geq \lm v$ outside of a compact set and which goes to infinity 
in all directions implies stochastic
completeness, see Corollary~6.6 in \cite{Gri99} for the manifold case and Proposition~5.5 in \cite{KLW13} for 
weighted graphs.
Furthermore, we note that, in the manifold setting, the equivalence of this formulation and stochastic completeness
was shown as Theorem~1.2 in \cite{MV13}.
\end{remark}

\section{Boundedness of geometry and of Laplacians}\label{s:bounded}
\subsection{Boundedness of the Laplacian}
We now discuss the boundedness of the Laplacian which turns out to be equivalent to boundedness of the
weighted vertex degree. Furthermore, it turns out that boundedness always implies stochastic completeness.

We start by a simple observation. We recall that $L$ is the self-adjoint operator on $\ell^2(X,m)$
which is obtained from the closure of the form $Q_c$ acting on $C_c(X) \times C_c(X)$ and that 
$\Deg(x) = 1/m(x) \sum_{y \in X}b(x,y)$ for $x \in X$ is the weighted degree of a vertex $x$. 
We first characterize
the boundedness of this operator in terms of the boundedness of the weighted degree function. This
fact is certainly well-known, see, for example Theorem~11 in \cite{KL10}.

\begin{theorem}[Boundedness of $L$]\label{t:boundedness_l2}
Let $G$ be a weighted graph. The Laplacian $L$ is a bounded operator on $\ell^2(X,m)$
if and only if $\Deg$ is bounded on $X$.
\end{theorem}
\begin{proof}
A direct calculation gives 
$$\langle L1_x, 1_x \rangle = \Deg(x)m(x)$$
where $1_x$ is the characteristic function of the set $\{x\}$ for $x \in X$.
Now, the result follows from the general theory of self-adjoint operators on Hilbert space, see for 
example Theorem~4.4 in \cite{Weid80}, by noting that $\{1_x/\sqrt{m(x)} \ | \ x\in X\}$ forms an
orthonormal basis for $\ell^2(X,m)$.
\end{proof}

Weighted graphs which satisfy the condition that $\Deg$ is a bounded function are sometimes referred
to as having \emph{bounded geometry}. We now have a look at this in the cases most commonly appearing
in the graph theory literature.

\begin{example} Let $G$ be a weighted graph.

\begin{enumerate}
\item If $m(x)=\sum_{y \in X} b(x,y)$ is the sum of the edge weights, then $\Deg(x)=1$
for all $x \in X$. Thus, in this case, $L$ is always a bounded operator.
\item If $G$ has standard edge weights, i.e., $b(x,y)\in \{0,1\}$ for all $x, y \in X$ and $m$ is the counting
measure, i.e.,  $m(x)=1$ for all $x \in X$, then
$$\Deg(x) = \# \{y \ | \ y \sim x \}$$
for all $x \in X$. Thus, $\Deg$ is just the usual vertex degree which counts the number of
neighbors of $x$. We see that $L$ is bounded in this case if and only if there is a uniform upper 
bound on this quantity.
\end{enumerate}
\end{example}

\subsection{Boundedness and stochastic completeness}
We now discuss the connection between boundedness and stochastic completeness.
In particular, we show that if $\Deg$ is bounded on $X$, then the graph is stochastically complete.
This follows from a more general result which allows for some growth of the weighted 
vertex degree which we state below.

\begin{theorem}[Boundedness implies stochastic completeness]\label{t:bounded_SC}
Let $G$ be a weighted graph. If for every infinite path $(x_n)_{n=0}^\infty$ 
$$\sum_{n=0}^\infty \frac{1}{\Deg(x_n)} = \infty,$$ 
then $G$ is stochastically complete.
In particular, if
$\Deg$ is bounded on $X$, then $G$ is stochastically complete.
\end{theorem}
\begin{proof}
By Theorem~\ref{t:sc_characterizations}~(iii), it suffices to show that any non-trivial $v \in \F$ with 
$\Lp v = \lm v$ for $\lm <0$ is not bounded. Suppose that $v(x_0)>0$ for some $x_0 \in X$. The equation
$\Lp v(x_0) = \lm v(x_0)$ can be rewritten as
$$\frac{1}{m(x_0)} \sum_{y \in X} b(x_0,y)v(y) = \left(\Deg(x_0)-\lm \right)v(x_0).$$
Hence, there exists $x_1 \sim x_0$ such that
$$v(x_1) \geq \left( 1 - \frac{\lm}{\Deg(x_0)} \right) v(x_0).$$
Now, we iterate this argument to get a sequence of vertices $x_0 \sim x_1 \sim x_2 \ldots$
such that
\begin{align*}
v(x_{n+1}) &\geq \left(1 - \frac{\lm}{\Deg(x_n)} \right)v(x_n) \\
&\geq \prod_{k=0}^n \left(1 - \frac{\lm}{\Deg(x_k)}\right) v(x_0).
\end{align*}
As $\sum_{k=0}^\infty 1/\Deg(x_k) = \infty$ if and only if $\prod_{k=0}^\infty (1 - \lm/\Deg(x_k))=\infty$,
it follows that $v$ cannot be bounded.
\end{proof}

\begin{remark} 
For an even shorter proof of the boundedness portion using the Omori-Yau maximum principle, see
Lemma~2.3 in \cite{Hua11a}. This is then extended to a boundendess of a notion of a global weighted degree
in Theorem~2.9 in \cite{Hua11a}. For a more precise result which only considers the maximal outward degree
on spheres in the case of standard edge weights and counting measure, see Theorem~4.2 in \cite{Woj11}
and Theorem~5.5 in \cite{Hua11a}.

When $b$ is the standard edge weight and $m$ is the counting measure, the boundendess result
was first shown via a minimum principle in \cite{DM06}. This proof was then extended to the case
of arbitrary edge weights and counting measure in \cite{Dod06}.

A more structural proof of the boundedness portion of Theorem~\ref{t:bounded_SC} can be 
found as Corollary~27 in \cite{KL10} and can be described as follows.  It turns out that 
the boundedness of $L$ on $\ell^2(X,m)$ also implies
boundedness of $\Lp$ acting on $\ell^\infty(X)$. In fact, the 
boundendess of $\Lp$ restricted to $\ell^p(X,m)$ for one $p \in [1,\infty]$ implies the boundedness of the restriction
of $\Lp$ to $\ell^p(X,m)$ for all $p \in [1,\infty]$. This was shown via the Riesz-Thorin interpolation
theorem as Theorem~9.3 in \cite{HKLW12} following earlier work presented as Theorem~11 in \cite{KL10}. Now, 
if $\Lp$ gives a bounded operator on $\ell^\infty(X)$,
then it is clear that the equation $\Lp v = \lm v$ cannot have a non-zero bounded solution for all $\lm <0$. 
Thus, by Theorem~\ref{t:sc_characterizations}~(iii), $G$ is stochastically complete.

We also note that the lower bound for the $\lm$-harmonic function appearing in the proof above 
can also be used to establish the essential self-adjointness
of the restriction of $\Lp$ to $C_c(X)$. See Proposition~2.2 in \cite{Gol14} 
or, more generally, Theorem~11.5.2 in \cite{Schm20}.
\end{remark}

\section{Weakly spherically symmetric graphs, trees and anti-trees}\label{s:symmetry}
\subsection{Weakly spherically symmetric graphs}
We now discuss a class of graphs for which we will give a full characterization of stochastic completeness.
These are weakly spherically symmetric graphs. They are an analogue to model manifolds which are
extensively discussed in \cite{Gri99}. The definition we give here was first presented in \cite{KLW13}
and later generalized in \cite{BG15}.

We start with some definitions. 
We assume that $G$ is connected and
recall that $d(x,y)$ denotes the combinatorial graph distance between
vertices $x$ and $y$, that is, the least number of edges in a path connecting $x$ and $y$. For a vertex
$x_0 \in X$ and $r \in \N_0$, we let $S_r(x_0)$ and $B_r(x_0)$ denote the sphere and ball of radius $r$
about $x_0$, that is,
$$S_r(x_0) = \{ x \in X \ | \ d(x,x_0)=r \}$$
and $ B_r(x_0) = \bigcup_{k=0}^r S_k(x_0) = \{ x \in X \ | \ d(x,x_0) \leq r \}.$
To ensure that these are finite sets, we now assume that $G$ is locally finite.

We will generally suppress the dependence on $x_0$ and just write $S_r$ and $B_r$.
We can then define the \emph{outer} and \emph{inner degrees} of a vertex $x \in S_r$ as
$$\Deg_{\pm}(x)  =\frac{1}{m(x)} \sum_{y \in S_{r \pm 1}} b(x,y).$$
That is, $\Deg_+(x)$ gives the total edge weight of edges going ``away'' from $x_0$ divided by the vertex
measure while $\Deg_-(x)$ of those going ``back'' towards $x_0$.

\begin{definition}[Weakly spherically symmetric graphs]
A locally finite connected weighted graph $G$ is called \emph{weakly spherically symmetric} 
if there exists a vertex $x_0 \in X$ such
that the functions $\Deg_\pm$ depend only on the distance to $x_0$. In this case, we will write
$\Deg_{\pm}(r)$ for $\Deg_{\pm}(x)$ when $x \in S_r(x_0)$.
\end{definition}
Again, although all of the concepts above depend on the choice of $x_0$, we will suppress this dependence
in our notation.  We note that this notion of symmetry is weak in the sense that we do not assume
anything about the edge weights between vertices on the same sphere nor do we assume anything about
the structure of the connections between vertices on successive spheres.

We will now state a full characterization for the stochastic completeness of such graphs. In order to do so, we introduce
the notion of \emph{boundary growth} of a ball as
$$\partial B(r) = \sum_{x \in S_r}\sum_{y \in S_{r+1}} b(x,y).$$
We note that $\partial B(r)$ reflects the total edge weight of edges leaving the ball $B_r$. Furthermore, 
for weakly spherically symmetric graphs, this can be written as
$$\partial B(r) = \Deg_+(r)m(S_r)= \Deg_{-}(r+1)m(S_{r+1})$$
as follows directly from the definitions. 
In particular, we note that
$$\partial B(r)= \frac{\partial B(r-1) \Deg_+(r)}{\Deg_-(r)}.$$
In what follows, we also let
$$V(r) = m(B_r)=\sum_{k=0}^r m(S_k)$$
denote the measure of a combinatorial ball of radius $r$.

\begin{theorem}[Stochastic completeness of weakly spherically symmetric graphs]\label{t:sc_symmetry}
If $G$ is a weakly spherically symmetric graph, then $G$ is stochastically complete if and only if
$$\sum_{r=0}^\infty \frac{V(r)}{\partial B(r)} =  \infty.$$
\end{theorem}
We give a sketch of the proof. For further details, see the proof of Theorem~5 in \cite{KLW13}.
For standard edge weights and counting measure, this was first shown as Theorem~4.8 in \cite{Woj11},
see also Theorem~5.10 in \cite{Hua11a} for an alternative proof in this case using the weak Omori-Yau maximum
principle.
\begin{proof}
By Theorem~\ref{t:sc_characterizations}~(iii)
it suffices to show that any bounded solution to $\Lp v = \lm v$ for $\lm<0$ is zero if and only if 
$\sum_{r=0}^\infty \frac{V(r)}{\partial B(r)} = \infty.$ By applying the characterization in 
terms of non-negative subsolutions in Theorem~\ref{t:sc_characterizations}~(iii$'$) and 
the Khas$'$minskii criterion from Theorem~\ref{t:Khas}, it suffices to consider only non-negative solutions $v$.
Finally, by averaging a solution over spheres, it suffices to consider only solutions depending on the
distance to $x_0$. 

Thus, we may write $v(r)$ for $v(x)$ for all $x \in S_r$ and note that stochastic completeness
is equivalent to the triviality of $v$ if $v$ is bounded. Now, by induction on $r \in \N_0$, it can
be shown by using the formulas above that  
$\Lp v(r) = \lm v(r)$ if and only if
$$v(r+1)-v(r) = \frac{-\lm}{\partial B(r)} \sum_{k=0}^r v(k)m(S_k).$$
In particular, if $v(0)>0$, then $v$ is strictly increasing with respect to $r$. Therefore, we estimate
$$ -\frac{\lm V(r)}{\partial B(r)} v(0) \leq v(r+1)-v(r) \leq -\frac{\lm V(r)}{\partial B(r)} v(r)$$
so that 
$$ v(r) - \frac{\lm V(r)}{\partial B(r)} v(0) \leq v(r+1) \leq \left(1 - \frac{\lm V(r)}{\partial B(r)} \right) v(r).$$
Iterating this down to $r=0$, gives 
$$ -\lm \sum_{k=0}^r \frac{V(k)}{\partial B(k)}  v(0) \leq v(r+1) \leq 
\prod_{k=0}^r \left(1 - \frac{\lm V(k)}{\partial B(k)} \right) v(0).$$
Hence, if $v$ is bounded, then $\sum_{k=0}^\infty \frac{V(k)}{\partial B(k)} <\infty$.
On the other hand, if $v$ is not bounded, then
$ \prod_{k=0}^\infty \left(1 - \frac{\lm V(k)}{\partial B(k)} \right) =\infty$ which is
equivalent to
$\sum_{k=0}^\infty \frac{V(k)}{\partial B(k)} =\infty$. This completes the proof.
\end{proof}

\subsection{Trees and anti-trees} 
We now illustrate the theorem above linking stochastic completeness of weakly spherically symmetric graphs 
and the ratio of the growth of the ball and the boundary of the ball with several examples. 
In particular, we introduce the class of spherically symmetric trees and anti-trees.

We start with spherically symmetric trees. 
For this, we first take standard edge weights and counting measure.
Such a graph $G$ is then called a \emph{spherically symmetric tree} if $G$ contains no cycles
and there exists a vertex $x_0 \in X$ such for all $x \in S_r$
$$\Deg_+(x) = \# \{ y \ | \  y \sim x, y \in S_{r+1} \}$$ 
only depends on $r$. Thus, we may write $\Deg_+(x)=\Deg_+(r)$ for
all $x \in S_r$.  Note that the lack of cycles implies that $\Deg_-(r)=1$ for all
$r \in \N$ so that the number of edges leading back to $x_0$ is minimal in order to 
have a connected graph.

We note that for spherically symmetric trees, we have
$$m(S_r) = \prod_{k=0}^{r-1} \Deg_+(k)$$ 
and $\partial B(r) = m(S_{r+1})$ as follows by direct calculations.
We now apply our characterization of stochastic completeness of weakly spherically
symmetric graphs to the case of such trees.

\begin{corollary}[Stochastic completeness and spherically symmetric trees]
If $G$ is a spherically symmetric tree, then $G$ is stochastically complete
if and only if
$$\sum_{r=0}^\infty \frac{1}{\Deg_+(r)} = \infty.$$
\end{corollary}
\begin{proof}
From the remarks directly above we obtain
$$\sum_{r=0}^\infty  \frac{V(r)}{\partial B(r)} = \sum_{r=0}^\infty \frac{1+\sum_{k=1}^r \prod_{j=0}^{k-1} \Deg_+(j)}{ \prod_{k=0}^{r} \Deg_+(k)}.$$
By the limit comparison test, it then follows that the divergence of the series above is equivalent to divergence
of the series $\sum_{r=0}^\infty 1/\Deg_+(r)$. Thus, the conclusions follows by Theorem~\ref{t:sc_symmetry}.
\end{proof}

The result above was first presented as Theorem~3.2.1 in \cite{Woj08}.   
It establishes the sharpness of the condition given for stochastic completeness in terms of the weighted
vertex degree on paths presented in Theorem~\ref{t:bounded_SC} in the previous section.

We note that the case of
spherically symmetric trees already provides a contrast with the manifold case as if we take
$V(r) = m(B_r)$ to be the counterpart of the volume growth in the Riemannian setting, then there
exist stochastically incomplete trees with factorial volume growth.
However, a much more striking example is that of anti-trees which we define next.
The basic idea is that we choose an arbitrary sequence of natural numbers for the number of 
vertices on the sphere and then connect all vertices between successive spheres. Thus, these
are the antithesis of trees in the sense that for trees the removal of a single edge between spheres
creates a disconnected graphs while for anti-trees one must remove all of the edges between spheres.

\begin{definition}[Anti-trees]
Let $(a_r)$ be a sequence with $a_r \in \N$ for $r \in \N$ and $a_0=1$.
A graph $G$ is called an \emph{anti-tree} with sphere growth $(a_r)$ 
if $G$ has standard edge weights and counting measure and the vertex set $X$ can be written
as a disjoint union $X= \bigcup_r A_r$ where $m(A_r)=a_r$ and $b(x,y)= b(y,x) = 1$ for all $x \in A_r, y \in A_{r+1}$
for $r \in \N_0$ and zero otherwise.
\end{definition}

Thus, by the definition of the edge weight, an anti-tree with sphere growth $(a_r)$  satisfies
$m(S_r)=a_r$ and is weakly spherically symmetric with $\Deg_{\pm}(r)=a_{r \pm 1}$. 
Furthermore,
$\partial B(r) = a_r a_{r+1}$ as each vertex in the sphere $S_r$ is connected to
all vertices in the sphere $S_{r+1}$. Therefore, we obtain the following characterization of stochastic completeness
in the case of anti-trees.
\begin{corollary}[Stochastic completeness and anti-trees]\label{c:antitree_sc}
If $G$ is an anti-tree with sphere growth $(a_r)$, then $G$ is stochastically complete
if and only if
$$\sum_{r=0}^\infty \frac{\sum_{k=0}^r a_k}{a_r a_{r+1}} = \infty.$$
\end{corollary}
\begin{proof}
This follows directly from the definition of an anti-tree and Theorem~\ref{t:sc_symmetry}.
\end{proof}

We note that if $a_r$ grows like $r^{2+\eps}$ for any $\eps>0$, then the corresponding anti-tree is
stochastically incomplete. Furthermore, $V(r)$ grows like $r^{3+\eps}$ in this case. Thus, unlike in the case
of manifolds, for the combinatorial graph metric, there exists stochastically incomplete graphs
with polynomial volume growth. We will also see later that these are the smallest such examples.
This motivates the move to different graph metrics which take into
account not only the combinatorial graph structure but also the vertex degree. 
These are the so-called intrinsic metrics which we introduce in the next section.

The result on stochastic incompleteness of anti-trees presented above originally appeared as 
Example~4.11 in \cite{Woj11}. To the best of our knowledge, the first example of 
an anti-tree in the special case of sphere growth $a_r=r+1$ appears as Example~2.5 in \cite{DK88}.
This anti-tree is a transient graph with the bottom of the spectrum at 0. The same graph
appears in \cite{Web10} as an example of a stochastically complete graph with unbounded
vertex degree.

\section{Intrinsic metrics}\label{s:intrinsic}
\subsection{A brief historical overview}
As we have seen, in order to hope for a counterpart for Grigor$'$yan's volume growth result
for graphs, we must go beyond the combinatorial graph distance when defining volume growth. In this section we introduce
the notion of an intrinsic metric for a weighted graph. This concept arises from Dirichlet form theory.
Although beyond the scope of this article, we mention that the form associated to the Laplacian,
which is a restriction of the graph energy form,
is a regular Dirichlet form which is not strongly local. For background on Dirichlet forms see \cite{FOT94},
for the connection between graphs and non-local regular Dirichlet forms see \cite{KL12}.
Furthermore, let us caution that the notion of an intrinsic metric for a Dirichlet form is distinct from the notion
of an intrinsic metric in the sense of length spaces as discussed, for example, in \cite{BBI01}.

The concept of an intrinsic metric for strongly local Dirichlet forms was brought into full fruition
in \cite{Stu94}. This allowed for the extension of a variety of results for Riemannian manifolds,
including Grigor$'$yan's volume growth result, to the setting of strongly local Dirichlet forms.
In particular, this covers the Riemannian setting as 
the Riemannian geodesic distance is an intrinsic metric for the strongly local
Dirichlet form arising in the manifold setting.
However, as mentioned above, the energy form of a graph is not strongly local so that the notions of \cite{Stu94}
do not cover the graph setting.

For non-local Dirichlet forms, such as particular restrictions of the energy form of a graph, the concept of an intrinsic
metric was discussed in full generality in \cite{FLW14}, see also \cite{MU11} as well as
\cite{Fol11, Fol14, GHM12} for the related notion of an adapted metric. 
However, as noted in \cite{FLW14}, the concept
of an intrinsic metric for a non-local form is more complicated than in the local setting
as the maximum of two intrinsic metrics is not necessarily an intrinsic metric. This can already be seen
in an easy example of a graph with three vertices, see Example~6.3 in \cite{FLW14}. The fact that
the maximum of two intrinsic metrics is an intrinsic metric for strongly local Dirichlet forms is essential
to establish the existence of a maximal intrinsic metric.

Thus, there does not exist a maximal intrinsic metric for graphs.
However, for proving statements in graph theory which are analogous to the strongly local setting, 
the tool of an intrinsic metric is 
quite useful, see the survey article \cite{Kel15} for an overview of results in this direction and further
historical notes and also \cite{HH} for some further recent applications.

\subsection{Intrinsic metrics, combinatorial graph distance and boundedness}
After this brief discussion of the history of intrinsic metrics, we now present the definition for graphs.
We call a function mapping pairs of vertices to non-negative real numbers 
a \emph{pseudo metric} if the map is symmetric, vanishes
on the diagonal and satisfies the triangle inequality. In other words, a pseudo metric is a metric
except for the fact that it might be zero for pairs of distinct vertices.
In general, intrinsic metrics are only assumed to be pseudo metrics. However, we will follow
convention and refer to them as metrics in any case.

\begin{definition}[Intrinsic metrics, jump size]
A pseudo metric $\vr:X \times X \lra [0,\infty)$ is called an \emph{intrinsic metric}
if
$$\sum_{y \in X} b(x,y) \vr^2(x,y) \leq m(x)$$
for all $x \in X$.
The quantity $j = \sup_{x \sim y} \vr(x,y)$ is called the \emph{jump size} of $\vr$. 
When $j < \infty$, we say that $\vr$ has \emph{finite jump size}.
\end{definition}

The use of intrinsic metrics often lies in a scenario when we want to estimate the energy
of a cut-off function defined with respect to an intrinsic metric via the measure. In the easiest example,
let $\vr$ be an intrinsic metric and let $\rho(x) = \vr(x,x_0)$ be the distance with respect to $\vr$
to a fixed vertex $x_0$. If $K \subseteq X$, then
\begin{align*}
\sum_{x,y \in K} b(x,y)(\rho(x)-\rho(y))^2 &= \sum_{x,y \in K} b(x,y)(\vr(x,x_0)-\vr(y,x_0))^2\\
&\leq \sum_{x \in K} \sum_{y \in K} b(x,y)\vr^2(x,y)\\
&\leq \sum_{x \in K} m(x) = m(K).
\end{align*}
In particular, we see that if $K$ is a set with finite measure, then the energy of $\rho$ on $K$ is finite.
The jump size becomes relevant whenever we have a cut-off function which is supported on $K$
and we need to control how far outside of the set $K$ the sum above reaches.

After this brief discussion, let us mention some examples. We recall that $d$ denotes the combinatorial graph 
distance, that is, the least number of edges in a path connecting two vertices. The case of when the
combinatorial graph distance is equivalent to an intrinsic metric can be characterized in terms
of the boundedness of the weighted vertex degree.
\begin{proposition}\label{p:intrinsic}
Let $G$ be a connected weighted graph. The combinatorial graph distance $d$ is equivalent to an 
intrinsic metric if and only if $\Deg$ is a bounded function on $X$.
\end{proposition}
\begin{proof}
We note that $d(x,y)=1$ for all $x \sim y$. Thus,
$$\sum_{y \in X} b(x,y)d^2(x,y) = \sum_{y \in X} b(x,y) = \Deg(x)m(x).$$
The conclusion now follows directly.
\end{proof}
Recall that by Theorem~\ref{t:boundedness_l2} this is the case exactly when the Laplacian is a bounded
operator and by Theorem~\ref{t:sc_symmetry}, the graph is stochastically complete in this case. 

Thus, we see that the combinatorial graph distance may or may not be intrinsic. We now look
at some further examples. In particular, the first example below gives a case when the combinatorial
graph distance is in fact intrinsic and the second gives a pseudo metric which is intrinsic for any given graph.
\begin{example}[Intrinsic metrics]\label{ex:intrinsic}
We now give two examples.
\begin{enumerate}
\item If $m(x)=\sum_{y\in X}b(x,y)$, then $\Deg(x)=1$ for all $x \in X$ and thus the combinatorial
graph distance $d$ is equivalent to an intrinsic metric by Proposition~\ref{p:intrinsic} directly above.
\item For a pair of neighboring vertices $x \sim y$ we let 
$$\si(x,y) = \left( \max \{ \Deg(x), \Deg(y) \} \right)^{-1/2}$$
denote the length of the edge connecting $x$ and $y$. Now, we can extend from the length of an edge to the
length of a path in a natural way, that is, if $(x_k)=(x_k)_{k=0}^n$ is a path, we let
$$l_\si((x_k)) = \sum_{k=0}^{n-1} \si(x_k, x_{k+1})$$
denote the length of the path.
Finally, we define a pseudo metric via
$$\vr_\si(x,y) = \inf \{ l_\si((x_k)) \ | \ (x_k) \textup{ is a path connecting } x \textup{ and } y\}.$$
As $\vr_\si(x,y) \leq \si(x,y)$ for $x \sim y$ it is then clear that $\vr_\si$ is an intrinsic metric as
$$\sum_{y \in X} b(x,y) \vr_\si^2(x,y) \leq \frac{1}{\Deg(x)} \sum_{y \in X} b(x,y) =m(x).$$
This metric was first introduced in \cite{Hua11b}, see also \cite{Fol11}. 

We note that if the graph is not
locally finite, then this intrinsic metric may be only a pseudo metric; however, in the locally finite case,
path metrics are metrics and give the discrete topology, see for example Lemma~A.3 in \cite{HKMW13}.
\end{enumerate}
\end{example}

The metric $\vr_\si$ introduced in the second example above shows that there always exists an intrinsic metric on a weighted
graph. This metric, which utilizes the weighted vertex degree function, makes sense in the context
of the process determined by the heat kernel. Namely, if the Markov process with transition probabilities
given by the heat kernel is a vertex $x$, then at the next jump time it moves with probability $b(x,y)/\sum_{z}b(x,z)$
to a neighbor $y$ of $x$. Furthermore, the wait time at the vertex $x$ is an exponentially distributed
random variable with parameter given by the weighted vertex degree, that is, the probability that 
the random walker is still at a vertex $x$ after time $t$ without having jumped is given by $e^{-\Deg(x)t}$. 
See Section~7 in \cite{KL10} for a further discussion of the connection between the heat semigroup and Markov
processes.
Therefore, at vertices with a large vertex degree, the process accelerates and thus will more quickly explore
neighboring vertices. Hence, for the process, neighbors of vertices of large vertex degree are close
which is consistent with the values of $\vr_\si$.

We note that given an intrinsic metric $\vr$, we can always obtain an intrinsic metric of small
jump size merely by cutting from above. That is, if $\vr$ is an intrinsic metric and $C>0$, then 
$$\vr_C(x,y) = \min \{ \vr(x,y), C\}$$
is also intrinsic with jump size at most $C$. On the other hand, having a uniform lower bound from below
on the distance between neighbors is equivalent to bounded geometry as we now show.
\begin{proposition}\label{p:intrinsic_2}
Let $G$ be a weighted graph. There exists an intrinsic metric $\vr$ such that $\vr(x,y) \geq C>0$
for all $x \sim y$ if and only if $\Deg$ is a bounded function on $X$.
\end{proposition}
\begin{proof}
If $\vr$ is an intrinsic metric with $\vr(x,y) \geq C>0$ for all $x\sim y$, then
$$C^2 \Deg(x)m(x) \leq \sum_{y \in X} b(x,y)\vr^2(x,y) \leq m(x)$$
so that $\Deg$ is bounded. 
Conversely, if $\Deg$ is bounded, then by Proposition~\ref{p:intrinsic},
it follows that the combinatorial graph distance is equivalent to an intrinsic metric $\vr$. In particular, there
exists a constant $C>0$ such that $C d(x,y) \leq \vr(x,y)$ for an intrinsic metric $\vr$. As $d(x,y)=1$
for all $x \sim y$, the conclusion follows.
\end{proof}

Thus, we obtain two conditions involving the existence of intrinsic metrics which imply stochastic completeness.
\begin{corollary}[Stochastic completeness and intrinsic metrics]
Let $G$ be a connected weighed graph. If either the combinatorial graph distance is equivalent
to an intrinsic metric or if there exists an intrinsic metric which is uniformly bounded below
on neighbors, then $G$ is stochastically complete.
\end{corollary}
\begin{proof}
This follows immediately by combining Propositions~\ref{p:intrinsic}~and~\ref{p:intrinsic_2}
with Theorem~\ref{t:bounded_SC}. 
\end{proof}

\subsection{A word about essential self-adjointness and metric completeness}\label{ss:metric_esa}
We briefly mention here some additional facts about metrics, geometry and analysis.
In Riemannian geometry, there is the famous Hopf-Rinow theorem, which gives a connection between
metric completeness, geodesic completeness and compactness of balls defined with respect to the geodesic metric,
see \cite{doCar92} for example.
For locally finite graphs, a counterpart is shown in \cite{HKMW13}, 
see also \cite{KM19} for a recent extension to a more general class of graphs. More specifically,
Theorem A.1 in \cite{HKMW13} shows that
for locally finite graphs and path metrics, the notions of metric completeness, geodesic completeness
in the sense that all infinite geodesics have infinite length, and finiteness of balls are equivalent.
Furthermore, if an intrinsic path metric on a locally finite graph
satisfies any of these equivalent conditions, then the restriction of the formal
Laplacian to the finitely supported functions is essentially self-adjoint, see Theorem~2 in \cite{HKMW13}. 
Thus, metric completeness with respect to an intrinsic metric
implies that there exists a unique Laplacian, at least for locally finite graphs. 
This corresponds to results known for Riemannian manifolds, see \cite{Che73, Stri83}.
For a more general and thorough discussion which includes this question for magnetic Schr{\"o}dinger operators on graphs
see \cite{Schm20}.

Subsequently, the assumption that there exists an intrinsic metric for which balls are finite has
often been used as a substitute for geodesic completeness in the graph setting. In particular, we will
see this assumption appearing in our criteria for stochastic completeness in the following sections.

\section{Uniqueness class, stochastic completeness and volume growth}\label{s:volume}
In this section, we will discuss the connections between uniqueness class results for the heat
equation, stochastic completeness and volume growth. As we have seen, using the combinatorial
graph metric gives a very different volume growth borderline for stochastic completeness compared to the manifold
setting. We will see that the results in these two settings can ultimately be reconciled via the use of intrinsic metrics.

\subsection{Uniqueness class}
We start by recalling Grigor$'$yan's uniqueness
class result on Riemannian manifolds. 
Specifically, if $u$ is a solution of the heat equation on a Riemannian manifold with zero
initial condition and if there
exists a monotone increasing function $f$ on $(0,\infty)$ such that $\int^\infty {r}/{f(r)}dr=\infty$
which dominates the growth of $u$ in the sense that
for all $r$ large
$$\int_0^T \int_{B_r} u^2(x,t)\  d\mu \ dt \leq e^{f(r)},$$
then $u=0$, see Theorem~9.2 in \cite{Gri99} or Theorem~11.9 in \cite{Gri09} for a proof. 
By Theorem~\ref{t:sc_characterizations} above
this immediately implies stochastic completeness under a suitable volume growth restriction.
More precisely, if $u$ is a bounded solution of the heat equation
with bound given by $C$, then letting $f(r)=\log(C^2 T V(r))$ shows that $u$ must be zero
whenever
$$\int^\infty \frac{r}{\log V(r)}dr=\infty.$$
Thus, the only bounded solution of the heat equation with trivial initial condition is trivial.

However, in \cite{Hua11b, Hua12}, there is already a counterexample to an analogue of this 
uniqueness class result when using intrinsic metrics. Namely, for the graph with
$X=\Z$, 
$$b(x,y)=\begin{cases}
1 & \textup{ if } |x-y|=1 \\
0 & \textup{ otherwise}
\end{cases}$$
and counting measure $m=1$, there exists an explicit function $u$ which is non-zero, satisfies the
heat equation with initial condition $0$ as well as the estimate
$$\int_0^T \sum_{x \in B_r}u^2(x,t) dt \leq  e^{f(r)}$$
for all large $r$ with $f(r)=\log T + C r \log r$ for some constant $C$. 
In particular, it is clear that
$$\int^\infty \frac{r}{f(r)}dr=\infty.$$
Thus, no analogue to Grigor$'$yan's uniqueness class result can hold for all graphs, even when using
intrinsic metrics.

We note that for this graph $\Deg(x)=2$ for all $x \in X$
and thus this graph is stochastically complete by Theorem~\ref{t:bounded_SC}. Therefore, this non-zero solution
cannot be bounded by Theorem~\ref{t:sc_characterizations}. Furthermore, we note by Proposition~\ref{p:intrinsic}
that the combinatorial graph metric is equivalent to an intrinsic metric in this case and that the volume
growth with respect to this metric is only quadratic. Finally, the constant $C$ appearing in the definition of $f$ 
in the example above is crucial
as for $C<1/2$, the uniqueness class result holds for all graphs, see Theorem~0.8 in \cite{Hua12}. 

The difference between these results in the discrete and continuous settings 
was ultimately resolved in \cite{HKS} by introducing a class of graphs for which a uniqueness
class result analogous to Grigor$'$yan's does hold.
These are the globally local
graphs which we introduce next. The idea for these graphs in the context of
a uniqueness class result is that the growth of the solution of the heat equation is balanced
by a decay in the jump size of an intrinsic metric as we go further out in the graph.
We note that in the counter example to the uniqueness class result above
we use the combinatorial graph distance whose jump size is always one and thus does not decay.

\begin{definition}[Globally local graphs]
Let $G$ be a weighted graph with a pseudo metric $\rho$. Let $B_r$ denote the ball of radius
$r$ with respect to $\rho$ and let $j_r$ denote the jump size of $\rho$ outside of $B_r$, that is,
$$j_r = \sup \{\rho(x,y) \ | \ x \sim y, \ x,y \not \in B_r\}.$$
$G$ is said to be \emph{globally local} in $\rho$ with respect to a monotone increasing function $f:(0,\infty) \lra (0,\infty)$
if $G$ has finite jump size, i.e., $j_0<\infty$ and if there exists a constant $A>1$ such that
$$\limsup_{r \to \infty} \frac{j_r f(Ar)}{r} < \infty.$$
\end{definition}

Thus, globally local graphs not only have finite jump size but provided that $f$ has a certain growth,
the jump size must decay outside of balls. In the borderline case for the uniqueness class result, 
$f$ is of the order $r \log r$ so that
$j_r$ must take care of the growth of $\log r$. We note, in particular, that the example mentioned above,
that the combinatorial metric used there as an intrinsic metric will not be globally local
with respect to $f(r)=r \log r$.

In any case, with this notion of globally local graphs, Theorem~1.3 in \cite{HKS} presents the following result.
\begin{theorem}[Uniqueness class for globally local graphs]\label{t:uniqueness_class}
Let $G$ be a weighted graph. Let $\vr$ be an intrinsic metric with finite balls $B_r^\vr$ and assume
that $G$ is globally local in $\vr$ with respect to a monotone increasing function $f:(0,\infty) \lra (0,\infty)$
such that
$$\int^\infty \frac{r}{f(r)}dr=\infty.$$
If $u:X \times [0,T] \lra \R$ is a solution of the heat equation with initial condition $0$ and
$$\int_0^T \sum_{x \in B_r^\vr} u^2(x,t) m(x) dt \leq e^{f(r)}$$
for all $r >0$, then $u=0$.
\end{theorem}
The proof of Theorem~\ref{t:uniqueness_class} can be found in Section~2 of \cite{HKS}. Though rather
long and technical, the main
idea is to estimate the size of a solution of the heat equation over a small ball at some time via the size of the
solution over a larger ball at an earlier time. Then one iterates this estimate down to time zero to show
that the solution must be trivial. Along the way, the use of cut-off functions involving intrinsic metrics
is crucial. This is, in part, because of the fact that in the discrete setting there is no good substitute
for the chain rule which is used throughout the proof of Grigor$'$yan's uniqueness class
result for manifolds.

\subsection{Stochastic completeness and volume growth in intrinsic metrics}
We want to use the uniqueness class result above to establish stochastic completeness under a volume
growth restriction which is valid for all graphs which allow for an intrinsic metric with finite distance balls. 
However, we note that the uniqueness class result above involves the additional assumption of being
globally local. Thus, some additional considerations are required in order
to reduce from general graphs to the class of globally local ones.

As a first step, it turns out that one can reduce to the case of an intrinsic metric with finite jump size 
via the notion of truncating the edge weights which
is already contained in \cite{GHM12}, see also \cite{MU11}. 
More specifically, if $\vr$ is an intrinsic metric for $G=(X,b,m)$, we define
new edge weights on $X$ via 
$$b_s(x,y) = \begin{cases}
b(x,y) & \textup{ if } \vr(x,y) \leq s \\
0 & \textup{ otherwise.}
\end{cases}$$
That is, we remove edges for which the distance between the adjacent vertices is large.
It follows that $\vr$ is also intrinsic for $G_s=(X,b_s,m)$ and $\vr$ now has finite jump size 
of at most $s$ on $G_s$.
Furthermore, Lemma~3.4 in \cite{HKS} gives that if $G_s$ is stochastically complete, then $G$ is stochastically
complete. We note that $G_s$ is not necessarily connected even if we start with a connected graph; 
however, the equivalent notions of stochastic 
completeness presented in Theorem~\ref{t:sc_characterizations} do not require connectedness of the graph.
As an alternate viewpoint, one may apply them on connected components of the graph. 
In particular, Lemma~3.4 in \cite{HKS}
uses the weak Omori-Yau characterization of stochastic completeness, that is, condition (iv) in 
Theorem~\ref{t:sc_characterizations} above to establish the result. 
See also Theorem~2.2 in \cite{GHM12} for a more general
statement involving Dirichlet forms associated to general jump processes. 

Thus, without loss of generality, we may assume that the intrinsic metric has finite jump size. 
The assumption of finite jump size along with finiteness of balls is easily seen to imply
local finiteness of the graph, see, for example Lemma~3.5 in \cite{Kel15}. Thus, we have reduced
to the case of locally finite graphs with finite jump size and finite distance balls.

Finally to reduce to the case of globally local graphs, the authors of \cite{HKS} use the notion of refinements
for locally finite graphs found in \cite{HS14}. The idea is to insert additional vertices within edges and extend the definitions
of the edge weights, vertex measure and the intrinsic metric in such a way that both the finiteness of balls
is preserved and that the measure of balls is only rescaled by a constant. Furthermore, as the inserted vertices
are now closer together with respect to the new intrinsic metric, it follows by Lemma~3.3 in \cite{HKS} 
that this can be done in such a way that
the refined graph is globally local with respect to an arbitrarily chosen function. Finally, Theorem~1.5 in \cite{HKS}
shows that stochastic completeness is preserved during the process of refining the graph via the use
of the weak Omori-Yau maximum principle.

Putting everything together, we get the following analogue to Grigor$'$yan's volume growth result which
can be found as Theorem~1.1 in \cite{HKS}.
\begin{theorem}[Volume growth and stochastic completeness]\label{t:sc_volume1}
Let $G$ be a weighted graph with an intrinsic metric $\vr$ with finite distance balls $B_r^\vr$.
Let $V_\vr(r) = m(B_r^\vr)$ and let $\log^{\#}(x) = \max\{\log(x), 1\}$. If
$$\int^\infty \frac{r}{\log^{\#} V_\vr(r)}dr =\infty,$$
then $G$ is stochastically complete.
\end{theorem}
\begin{proof}[Sketch of proof]
From the discussion above, we can reduce to the case of finite jump size and finite balls and, thus, 
to locally finite graphs. In this case, the technique of refinements allows us to reduce to the case of graphs which are
globally local with respect to $f(r)=\log^{\#}V_\vr(r)$. Thus, given a bounded solution of the heat
equation $u$ with initial condition 0 and bound $C$, we obtain
$$\int_0^T \sum_{x \in B_r^\vr}u^2(x,t) m(x) dt \leq C^2T e^{f(r)} = e^{f(r)+C_1}$$
for some constant $C_1$.
Therefore, by Theorem~\ref{t:uniqueness_class}, we obtain that $u=0$ and by Theorem~\ref{t:sc_characterizations}~(ii$'$)
we get that $G$ is stochastically complete.
\end{proof}

\begin{remark}
\begin{enumerate}
\item We note that the use of $\log^{\#}$ instead of just $\log$ is to deal with the case when the measure
$m$ is small. In particular, this covers the case when the entire vertex set has finite measure, that
is, $m(X)<\infty$.
In this case, stochastic completeness is actually equivalent to two other properties, namely, 
to recurrence and to form uniqueness, see Theorem~16
in \cite{Schm17} or Theorem~7.1 in \cite{GHKLW15} for further details.  Thus, we see that in the case of finite measure,
the existence of an intrinsic metric which gives finite balls implies all three of these properties. A partial
converse to this result was recently proven in \cite{Puch}. More specifically, if a graph is recurrent, then 
there exists a finite measure and an intrinsic metric which has finite distance balls.
For a precise statement, see Theorem~11.6.15 in \cite{Schm20}.

\item In the case of locally finite graphs with counting measure and an intrinsic metric with finite jump size
and finite balls, 
Theorem~\ref{t:sc_volume1} was first
shown as Theorem~1.2 in \cite{Fol14}. However, not only does the formulation have additional assumptions
on the graph,
but the proof is quite different from Grigor$'$yan's original proof on manifolds. Namely, the approach in \cite{Fol14} is to synchronize the random
walk on the discrete graph with a random walk on a metric graph. Metric graphs are graphs where edges are 
intervals of real numbers. In particular, the energy form on a metric graph is a strongly local Dirichlet form. 
Thus, the extension of Grigor$'$yan's result to strongly local Dirichlet forms shown in \cite{Stu94} implies
stochastic completeness given that the volume growth of the discrete graph is comparable to the volume
growth of the metric graph. We note that the proof in \cite{Stu94} also does not invoke the heat equation.

An analytic proof of the result in \cite{Fol14} using the Omori-Yau maximum principle and 
which allows for an arbitrary vertex measure but still uses metric graphs and assumes local finiteness
and finite jump size can be found in \cite{Hua14b}. Thus,
while an analogue of Grigor$'$yan's result was known for some classes of graphs since \cite{Fol14}, the paper
\cite{HKS} contains the first proof which does not assume finite jump size nor local finiteness and does not
invoke metric graphs. 
\end{enumerate}
\end{remark}

\subsection{Stochastic completeness and volume growth in the combinatorial graph metric} 
We now briefly discuss how in the case of standard edge weights and counting measure the paper \cite{GHM12}
already contains the optimal growth rate result for stochastic completeness when using the combinatorial
graph distance. More specifically, the authors of \cite{GHM12} first use the method of \cite{Dav85} from the
strongly local setting to establish that the volume growth condition
$$\liminf_{r \to \infty} \frac{\log V_{\vr_1}(r)}{r \log r}<\frac{1}{2}$$
implies stochastic completeness of general jump processes, see also \cite{MUW12}
where it is shown that the $1/2$ on the right hand side can be replaced by $\infty$. 
Here, the volume growth is defined with
respect to what the authors of \cite{GHM12} call an \emph{adapted} metric. 
The idea for an adapted metric in \cite{GHM12} is that the intrinsic condition
needs to be only satisfied for pairs of vertices that are close with respect to $\vr$, that is, 
one truncates the metric before imposing the intrinsic condition.
More specifically, letting $\vr$ be a pseudo metric and $\vr_1 =\min \{ \vr, 1\}$, then
$\vr_1$ must satisfy
$$\sum_{y \in X}b(x,y)\vr_1^2(x,y) \leq m(x)$$
for all $x \in X$ in order for $\vr$ to be called adapted.

Clearly any intrinsic metric is adapted. 
On the other hand, we can also modify Example~\ref{ex:intrinsic}~(2) in a natural way
to get an adapted metric. 
More specifically, by defining the length of an edge now by 
$$\si_1(x,y) = \left( \Deg(x) \vee \Deg(y) \right)^{-1/2} \wedge 1$$ 
where $a \vee b = \max \{a, b \}$ and $a \wedge b = \min \{a,b\}$ 
then we can extend to paths to get an adapted metric $\vr_{\si_1}$. Then, for locally finite graphs,
Corollary~4.3 in \cite{GHM12} gives that if $\inf_{x \in X} m(x)>0$ and
$$V_{\vr_{\si_1}}(r) \leq e^{Cr}$$ 
for $C>0$ and all large $r$, then $G$ is stochastically complete. 
We note that the additional assumption that $\inf_{x \in X} m(x)>0$
is used along with the volume growth assumption to show that balls defined with respect to $\vr_{\si_1}$ are finite.
Thus, these assumptions are used as a replacement for the finiteness of balls assumption found 
in Theorem~\ref{t:sc_volume1} above.

Although clearly not optimal when compared with Theorem~\ref{t:sc_volume1},
it turns out that this volume growth result already gives an optimal volume growth condition in the case of standard edge
weights, counting measure and combinatorial growth distance. More specifically, letting $B_r$ denote the ball
of radius $r$ defined with respect to the combinatorial graph metric $d$ and $V(r)=m(B_r)$ which 
is the number of vertices in the ball in this case, 
Theorem~1.4 in \cite{GHM12} gives the following result.

\begin{theorem}\label{t:sc_volume_combinatorial}
Let $G$ be a graph with standard edge weights and counting measure.
If 
$$V(r) \leq C r^3$$
for some constant $C>0$ and all large $r$, then $G$ is stochastically complete.
\end{theorem}
\begin{proof}[Idea of proof]
It can be shown that under the assumption $V(r) \leq C r^3$, there
are sufficiently many vertices with small degree so that a ball defined with respect $\vr_{\si_1}$ is contained
in a ball of larger radius with respect to the combinatorial graph distance $d$. 
In particular, the volume growth restriction on $V(r)$ can be used to get a volume growth restriction
on $V_{\vr_{\si_1}}(r)$ and then stochastic completeness follows from the volume growth
criterion for $V_{\vr_{\si_1}}(r)$ mentioned above.
For more details, see Section~4
in \cite{GHM12}.
\end{proof}

We note that the characterization of stochastic completeness of anti-trees provided in Corollary~\ref{c:antitree_sc} above
gives the sharpness of Theorem~\ref{t:sc_volume_combinatorial} as for an anti-tree with sphere growth 
$a_r$ of the order $r^{1+\eps}$ for any $\eps>0$, we have that $V(r)$ grows like $r^{3+\eps}$ and that the anti-tree is stochastically incomplete. Furthermore, in this case the weighted degree of vertices grows like
$r^{2+\eps}$ so that balls defined with respect to $\vr_{\si_1}$ are not finite and thus the graph is not geodesically
complete. Finally, we note that
Theorem~\ref{t:sc_volume_combinatorial} remains valid in the more general setting where the edge weights
and vertex measure satisfy
$$b(x,y) \leq C m(x)m(y)$$
for all $x, y \in X$ and some $C>0$, see Remark~4.4 in \cite{GHM12}. This type of assumption is sometimes
called an \emph{ellipticity condition} on graphs, for more details and some applications of this condition
in the context of curvature see \cite{KM}.

\section{Stochastic completeness and curvature}\label{s:curvature}
In recent years there has been a surge of interest in various notions of curvature in the discrete setting.
We do not even attempt to give a comprehensive overview of definitions nor of results. We rather confine
ourselves to two of the most prominent definitions of curvature and discuss results which relate curvature and stochastic 
completeness. The two notions of curvature that we will discuss are that of Bakry--{\`E}mery 
which arises from the $\Gamma$-calculus as outlined in \cite{BE85}
and Ollivier Ricci which arises from optimal transportation theory and is defined for general Markov
chains in \cite{Oll09}. 
We will briefly introduce each
and present the relevant results for stochastic completeness.

We note that, unlike in the manifold case where the Bishop--Gromov inequality relates lower bounds on Ricci curvature
to volume growth, there is no analogous connection between lower curvature bounds and volume growth in
full generality in the discrete setting thus far. 
Therefore, we are not able to relate the 
volume growth criteria for stochastic completeness presented in the previous section to the curvature
conditions given in this section.
However, for Bakry--{\'E}mery curvature, see Theorem~1.8 in \cite{HM} for some recent progress in connecting 
curvature and volume growth for a specific class of graphs and Theorem~4.1 in \cite{M} for
a connection between lower curvature bounds and the volume doubling property for finite
graphs. Furthermore, \cite{AS} establishes comparisons between averaged inner and outer
degrees and volume growth and also gives an example of two graphs 
with equal Olliver Ricci curvatures but different volume growths. 

\subsection{Bakry--{\'E}mery curvature and stochastic completeness}
We start with Bakry--{\'E}mery curvature. This notion has origins in work on hypercontractive
semigroups found in \cite{BE85}. For early manifestations in the graph setting, see \cite{LY10,Schm96}. 
We caution the reader at the outset that in the curvature on graphs community, one usually takes the Laplacian
with the opposite sign of ours.

We first introduce the $\Gamma$-calculus. In order to take care of convergence of sums, we now assume
that all graphs are locally finite. In this case, we note that the domain of the formal Laplacian is the set of all functions
on $X$, that is, $\F = C(X)$. For $f, g \in C(X)$ and $x \in X$, we let
$$\Gamma(f,g)(x) = -\frac{1}{2} \left(\Lp(fg)-f \Lp g - g \Lp f\right)(x).$$
We will follow convention and write $\Gamma(f)$ for $\Gamma(f,f)$. 
By a direct calculation we then obtain
$$\Gamma(f)(x)=\frac{1}{2m(x)}\sum_{y \in X}b(x,y)(f(x)-f(y))^2.$$
In particular, if $\Gamma(f)=0$, then $f$ is constant on any connected component
of the graph. In some sense, $\Gamma$ can be thought of as an analogue to the norm squared of the gradient
from the continuous setting.

We then define
$$\Gamma_2(f)= -\frac{1}{2}\Lp \Gamma(f) + \Gamma(f, \Lp f).$$
With these notations $G$ is said to satisfy 
$CD(K,\infty)$ at $x \in X$ for $K \in \R$ if
$$\Gamma_2(f)(x) \geq K \Gamma(f)(x)$$
for all $f \in C(X)$. The 
$CD(K,\infty)$ condition on $G$ is then just the fact
that $G$ satisfies the conditions at all $x \in X$.

\begin{definition}
A locally finite weighted graph $G$ is said to satisfy 
$CD(K,\infty)$ for $K \in \R$ if
$$\Gamma_2(f) \geq K \Gamma(f)$$
for all $f \in C(X)$.
\end{definition}

The idea of the definition is to mimic the inequality obtained via Bochner's formula and a lower
Ricci curvature bound in the Riemannian manifold setting. The number $K$ is then thought
to be a lower curvature bound on the graph.

We now work towards giving criteria for stochastic completeness involving Bakry--{\'E}mery curvature
conditions. We start with the main result found as Theorem~1.2 in \cite{HL17} which gives stochastic completeness
under the condition $CD(K,\infty)$, finiteness of distance balls and a uniform lower bound on the measure. 
We recall that we denote the heat semigroup by 
$$P_t=e^{-tL}$$
for $t \geq 0$. This semigroup is originally defined on $\ell^2(X,m)$ via the spectral theorem and can then be extended to
all $\ell^p(X,m)$ spaces for $p \in [1,\infty]$ via monotone limits. 
The heat semigroup is also strongly continuous, that is,
$P_t f \to f$ as $t \to 0^+$ for all $f \in \ell^\infty(X)$.
Furthermore, $P_t$ is Markov, specifically,
for $0 \leq f \leq 1$, we have $0 \leq P_t f \leq 1$. Stochastic completeness is then equivalent to the 
fact that $P_t 1 =1$ for all $t \geq 0$. 

\begin{theorem}[Stochastic completeness and Bakry--{\'E}mery curvature]\label{t:sc_BE}
Let $G$ be a locally finite graph connected with $\inf_x m(x) >0$ 
and an intrinsic metric $\vr$ with finite distance balls.
If $G$ satisfies $CD(K,\infty)$, then $G$ is stochastically complete.
\end{theorem}
\begin{proof}[Idea of proof]
The bulk of the work is in showing that $CD(K,\infty)$ is actually equivalent to the 
following gradient estimate on the heat semigroup
$$ \Gamma (P_t \vph) \leq e^{-2Kt}P_t(\Gamma(\vph))$$
for all $\vph \in C_c(X)$ and $t \geq 0$, see Theorem~4.1 in \cite{HL17}. 
Furthermore, the finiteness of balls
with respect to an intrinsic metric allows for the construction of a sequence $\vph_n \in C_c(X)$ such that
$0 \leq \vph_n \leq 1$, $\vph_n(x) \to 1$ as $n \to \infty$ for every $x \in X$ and such that
$$\Gamma(\vph_n) \leq \frac{1}{n}$$
for all $n \in \N$.
In particular, one lets
$$\vph_n(x) = \left(\frac{2n-\vr(x, x_0)}{n}\right) \vee 0 \wedge 1$$
where $x_0 \in X$ is an arbitrary vertex, $a\vee b = \max \{a,b\}$ and $a\wedge b = \min\{a, b\}$.
It is then clear from the definition that $0 \leq \vph_n \leq 1$, $\vph_n =1$ on $B_n$, $\vph_n$ is supported
on $B_{2n}$, and thus $\vph_n \in C_c(X)$ by the assumption of finite distance balls,
and using the fact that $\vr$ is intrinsic, a direct calculation gives $\Gamma(\vph_n) \leq \frac{1}{n}$.

As the semigroup is Markov on $C_c(X)$, we then obtain 
$$P_t (\Gamma(\vph_n)) \leq \frac{1}{n}.$$
Given these ingredients, the proof is now straightforward as
$P_t \vph_n(x) \to P_t 1(x)$ for all $x \in X$ by the monotone convergence theorem and thus
\begin{align*}
\Gamma(P_t 1)(x) &= \lim_{n \to \infty} \Gamma(P_t \vph_n)(x) \leq \liminf_{n \to \infty} e^{-2Kt}P_t(\Gamma(\vph_n))(x) \\
&\leq \liminf_{n \to \infty} e^{-2Kt} \frac{1}{n} = 0
\end{align*}
for all $x \in X$. 
Thus, $P_t 1$ is constant for every $t \geq 0$. 
From the heat equation we obtain
$$\pt P_t 1 = - \Lp P_t 1 =0$$
for all $t >0$.
As $P_0 1 =1$ and the semigroup is strongly continuous
we obtain that $P_t 1 = 1$ for all $t \geq 0$.
\end{proof}

\begin{remark}
We note that Theorem~\ref{t:sc_BE} includes not only the assumption of finiteness of balls but also
the additional assumption that $\inf_x m(x)>0$. Both of
these assumptions appear in the context of essential self-adjointness.
In particular, $\inf_x m(x)>0$ implies both that $\Lp$ maps $C_c(X)$ 
into $\ell^2(X,m)$ and that the restriction of $\Lp$ to $C_c(X)$ is
essentially self-adjoint, see Theorem~6 in \cite{KL12}. In the context of the proof of 
Theorem~\ref{t:sc_BE} in \cite{HL17}, this assumption is used to establish the convergence of sums in the 
proof of the estimate $ \Gamma (P_t \vph) \leq e^{-2Kt}P_t(\Gamma(\vph))$. However, it is not
clear whether this assumption is really necessary for the ultimate result of stochastic completeness.

On the other hand, the result in \cite{HL17} actually assumes a seemingly weaker assumption than finiteness
of balls found in the formulation of Theorem~\ref{t:sc_BE} above. Namely, Theorem~1.2 in \cite{HL17} only
assumes the existence of a sequence of finitely supported functions $(\vph_n)$ 
such that $0 \leq \vph_n \leq 1,
\vph_n(x) \to 1$ and $\Gamma(\vph_n) \leq 1/n$ for all $n \in \N$. 
This is sometimes referred to as a \emph{completeness} assumption
on the graph as, in the manifold setting, the existence of such a sequence 
is known to be equivalent to geodesic completeness, see \cite{BGL14,
Stri83}. It is clear in the proof above that the existence of an intrinsic metric with finite distance balls
implies the existence of such a sequence. On the other hand,
Marcel Schmidt recently communicated to us that the converse is also true, that is, that
the existence of an intrinsic metric with finite distance balls is actually equivalent to completeness as defined above.
For a proof of this fact, see Appendix~A in the updated version of \cite{LSW}. Furthermore, for graphs satisfying
the ellipticity condition $b(x,y) \leq C m(x)m(y)$ for all $x, y \in X$ and which are
stochastically complete and satisfy an additional condition called the Feller property, see \cite{PS12, Woj17} 
for more details, $CD(0,\infty)$ implies that the graph is complete, see Theorem~6.1 in \cite{KM}.
\end{remark}

Theorem~\ref{t:sc_BE} gives stochastic completeness in the case of a uniform lower Bakry--{\'E}mery curvature 
bound in the spirit of \cite{Yau78}.
However, looking at the optimal curvature results from the Riemannian setting mentioned in the introduction,
we would expect to allow for some decay of Ricci curvature as in \cite{Var83, Hsu89}. One improvement of
Theorem~\ref{t:sc_BE} in this direction is contained for a special class of graphs in \cite{HM}. More specifically,
if $G$ is a graph with $X = \N_0$, $x \sim y$ if and only if $|x-y|=1$, $m$ is the counting measure
and $\vr$ is an intrinsic metric on $G$, then letting
$$\kap(x) = \sup \{K \in \R \ | \ G \textup{ satisfies } CD(K,\infty) \textup{ at } x \}$$
if $\kap(x)$ decays like $-\vr^2(0,x)$, then $G$ is stochastically complete,
see Theorem~1.6 in \cite{HM} for a more precise statement and proof.

\subsection{Ollivier Ricci curvature and stochastic completeness}
We now discuss a second commonly appearing manifestation of curvature in the discrete setting. 
This formulation comes from optimal transport theory, see \cite{Vill03} for a general background, and was defined for
Markov chains in \cite{Oll07, Oll09}. For graphs, the basic idea is to transport a mass, which is given
by the transition probability of a simple random walker starting at a vertex, to that of a mass at another vertex 
with minimal effort. This definition
was then modified to give an infinitesimal version for the case of bounded degree in \cite{LLY11} and extended
to graphs with general measure and edge weights in \cite{MW19}. 
We note that, in contrast to the Bakry--{\'E}mery formulation which
is defined at a vertex, this curvature is defined for pairs of vertices.

We start with some basic definitions. As in the previous subsection, we assume that all graphs are locally finite.
We will further assume that all graphs are connected.
For a vertex $x \in X$ and $\eps>0$ small, one defines a transition probability distribution $\mu_x^\eps$ on $X$ via
$$\mu_x^\eps(y)= 
\begin{cases}
1 - \eps \Deg(x) & \textup{ if } y =x \\
\eps b(x,y)/m(x)& \textup{ otherwise}
\end{cases}
$$
where $\Deg(x)$ is the weighted degree of $x$.
This is a probability distribution provided that $\eps \leq 1/\Deg(x)$. We note that a connection to the
Laplacian is given by
$$\mu_x^\eps(y)= 1_y(x) -\eps L 1_y(x)$$
as follows by a direct calculation.

Now, for two vertices $x_1, x_2 \in X$, we define the Wasserstein distance between $\mu_{x_1}^\eps$
and $\mu_{x_2}^\eps$ via
$$W(\mu_{x_1}^\eps, \mu_{x_2}^\eps) = \inf_{\pi} \sum_{x,y \in X}\pi(x,y) d(x,y)$$
where the infimum is taken over all $\pi:X \times X \lra [0,1]$ with
$\sum_{y \in X}\pi(x,y) = \mu_{x_1}^\eps(x)$ and $\sum_{x \in X}\pi(x,y) = \mu_{x_2}^\eps(y)$
and $d$ is the combinatorial graph distance. The idea behind this is that $\pi$ transports the mass
distribution from $\mu_{x_1}^\eps$ to $\mu_{x_2}^\eps$ and thus $W(\mu_{x_1}^\eps, \mu_{x_2}^\eps)$
minimizes the effort required to carry out this transport.

We let $Lip(1)$ denote the set of functions with Lipshitz constant 1 with respect to the combinatorial
graph distance, that is,
$$Lip(1) = \{ f \in C(X) \ | \ |f(x)-f(y)| \leq d(x,y) \textup{ for all } x, y \in X \}$$
and let $\ell^\infty(X)$ denote the set of bounded functions.
By Kantorovich duality, see Theorem~1.14 in \cite{Vill03}, we have
$$W(\mu_{x_1}^\eps, \mu_{x_2}^\eps) = 
\sup_{f \in Lip(1) \cap \ell^\infty(X)} \sum_{x \in X}f(x)(\mu_{x_1}^\eps(x)-\mu_{x_2}^\eps(x)).$$
Finally, following \cite{Oll09, LLY11}, in \cite{MW19} we define the Ollivier Ricci curvature
between two vertices as follows.

\begin{definition}[Ollivier Ricci curvature]
For vertices $x, y \in X$ with $x \neq y$, we let 
$$\kap^\eps(x,y)=1- \frac{W(\mu_x^\eps, \mu_y^\eps)}{d(x,y)}$$
and define the \emph{Ollivier Ricci curvature} as
$$\kap(x,y)= \lim_{\eps \to 0^+} \frac{\ka^\eps(x,y)}{\eps}.$$
\end{definition}

\begin{remark}
We note that \cite{Oll09} often considers the case when $m(x)=\sum_{y \in X}b(x,y)$
and $\eps=1$. In this case, $\Deg(x)=1$ so that
$$\mu_x^1(y) = \frac{b(x,y)}{\sum_{z \in X} b(x,z)}$$
is just the one step transition probability at $x$ of the simple random walk on $G$.
More generally, when $\eps<1$, we get that
$$\mu_x^\eps(y) = \begin{cases}
1 - \eps & \textup{ if } y =x \\
 {\eps b(x,y)}/{\sum_{z \in X} b(x,z)} & \textup{ otherwise}.
\end{cases}$$
In this case, the simple random walk is given a positive probability to remain at the vertex $x$. As such,
the constant $1- \eps$ is sometimes referred to as the idleness parameter in this setting,
see \cite{BCLMP18}.

The idea of letting $\eps \to 0^+$ in the case when $m(x) = \sum_{y \in X}b(x,y)$ then
appears in \cite{LLY11}. In particular, the concavity of the function $\kap^\eps$ is used
to establish the existence of the limit. However, as noted previously, the assumption
$m(x) = \sum_{y \in X}b(x,y)$ gives that $\Deg(x)=1$ and thus all such graphs are stochastically
complete by Theorem~\ref{t:bounded_SC}. 
The contribution of \cite{MW19} is to allow for this definition in the case of
possibly unbounded vertex degree which makes the question of stochastic completeness
interesting. The existence of the limit as $\eps \to 0^+$ follows analogously to the argument
in \cite{LLY11}, see also \cite{BCLMP18}.
\end{remark}

We note that the Ollivier Ricci curvature as defined above can be explicitly calculated in many cases.
For example, if the vertices $x\sim y$ are not contained in any 3-, 4-, or 5-cycles, then
$$\kap(x,y) = 2b(x,y)\left( \frac{1}{m(x)}+\frac{1}{m(y)}\right) - \Deg(x) - \Deg(y)$$
see Example~2.3 in \cite{MW19}. As a concrete illustration, in the case of standard edge weights
and counting measure, the curvature of an edge in a $k$-regular tree is 
$\kap(x,y)=4-2k$. This confirms the notion that regular trees of degree greater than 2 are the analogues
of hyperbolic space as they have constant negative curvature.

As a second example, let $X=\N_0$ and let $x\sim y$ if and only if $|x-y|=1$.
Such graphs will be referred to as \emph{birth-death chains}.
We note that they automatically fall into the framework of weakly spherically symmetric graphs as 
discussed in Section~\ref{s:symmetry}.
Let $j, k \in \N_0$ with $k > j$. The Ollivier Ricci curvature on a birth-death chain can be calculated directly as
$$\kap(j, k) = \frac{1}{k-j} \left( \frac{b(j, j+1)-b(j,j-1)}{m(j)} - \frac{b(k,k+1)-b(k,k-1)}{m(k)}\right)$$
see Theorem~2.10 in \cite{MW19}.  We note that for $j \sim k$, this reduces to the formula above.

Both of the examples above are easily derived from the following
formula which allows us to compute the curvature in terms of the Laplacian. 
To state it, we let
$$\nabla_{xy}f= \frac{f(x)-f(y)}{d(x,y)}$$
for $x \neq y$.
For a proof of the following formula see Theorem~2.1 in \cite{MW19}.
\begin{theorem}[Ollivier Ricci curvature and Laplacian]\label{t:curvature_Lap}
Let $G$ be a locally finite connected weighted graph. For vertices $x \neq y$, we have
$$\kap(x,y) =\inf_{\substack{f \in Lip(1) \\ \nabla_{xy}f=1}} \nabla_{xy}\Lp f
 =\inf_{\substack{f \in Lip(1) \cap C_c(X) \\ \nabla_{xy}f=1}} \nabla_{xy}\Lp f.$$
\end{theorem}

We now show how Theorem~\ref{t:curvature_Lap} allows us to prove a Laplacian
comparison result. In this result, we compare the Laplacian applied to a distance function
on a general graph to the Laplacian applied to a distance function on a birth-death chain.
We first define the notion of sphere curvature which will be involved. We recall that for a vertex $x_0 \in X$, $S_r$ denotes
the sphere of radius $r$ around $x_0$ with respect to the combinatorial graph distance.
We then let
$$\kap(r) = \min_{y \in S_r} \max_{\substack{x \in S_{r-1} \\ x \sim y}}\kap(x,y)$$
for $r\in \N$ with $\kap(0)=0$ and call $\kap$ the \emph{sphere curvature}.

\begin{theorem}[Laplacian comparison]\label{t:Lap_comp}
Let $G$ be a locally finite connected weighted graph. If $x_0 \in X$, $\rho(x)=d(x,x_0)$ and $\kap$ denotes
the sphere curvature, then
$$\Lp \rho(x) \geq \sum_{j=1}^{\rho(x)} \kap(j) - \Deg(x_0)$$
for all $x \in X$.
\end{theorem}
\begin{proof}
We note that in general
$$\Lp \rho(x) = \Deg_-(x)-\Deg_+(x)$$
where $\Deg_\pm(x)$ are the outer and inner degrees as defined in Section~\ref{s:symmetry}. 
In particular, $\Lp \rho(x_0)=-\Deg(x_0)$ so taking the
sum to be zero in this case gives the statement for $x=x_0$.

The proof is now by induction on $r=\rho(x)$ for $x \in S_r$. Assume that the statement is true for $r-1$. 
Let $y \in S_r$ and let $x\in S_{r-1}$
be such that $\kap(x,y) \geq \kap(x,z)$ for all $z \sim x$ with $z \in S_r$. We note that
$\kap(r) \leq \kap(x,y)$, $\nabla_{yx}\rho=1$ and $\rho \in Lip(1)$ 
so that Theorem~\ref{t:curvature_Lap} gives 
\begin{align*}
\kap(r) \leq \kap(x,y) &\leq \nabla_{yx}\Lp \rho =\Lp \rho(y) - \Lp \rho(x).
\end{align*}
Therefore, by the inductive hypothesis,
$$\sum_{j=1}^{r} \kap(j) - \Deg(x_0) = \sum_{j=1}^{r-1} \kap(j) - \Deg(x_0) + \kap(r) \leq \Lp \rho(x) + \kap(r) \leq \Lp \rho(y)$$
which completes the proof.
\end{proof}

We note that the Laplacian comparison result is sharp on birth-death chains as
can be seen by a direct calculation. That is, for birth-death chains, we get
$$\Lp \rho(x) = \sum_{j=1}^{\rho(x)} \kap(j) - \Deg(x_0).$$
Thus, Theorem~\ref{t:Lap_comp} compares the Laplacian of a distance function
on a general graph to that of a birth-death chain.

We now put the various pieces together to obtain a stochastic completeness result for Ollivier Ricci
curvature. This can be found as Theorem~4.11 in \cite{MW19}.
\begin{theorem}\label{t:SC_Oll_curv}
Let $G$ be a locally finite connected weighted graph. If
$$\kap(r) \geq -C \log r$$
for some constant $C>0$ and all large $r$, then $G$ is stochastically complete.
\end{theorem}
\begin{proof}
Let $\rho(x)=d(x,x_0)$. It follows from Theorem~\ref{t:Lap_comp} that
$$\Lp \rho(x) \geq \sum_{j=1}^{\rho(x)} \kap(j) - \Deg(x_0).$$
Now, from the assumption that $\kap(r) \geq -C \log r$ we may choose an
increasing continuously differentiable function $f:[0,\infty) \lra (0,\infty)$ such that
$$\Lp \rho + f(\rho) \geq 0$$
and 
$$\int^\infty \frac{1}{f(r)}dr=\infty.$$
As $\rho(x_n) \to \infty$ along any sequence of vertices $(x_n)$ with
$\Deg(x_n) \to \infty$, $G$ is stochastically complete by Theorem~\ref{t:Khas}.
\end{proof}

It turns out that Theorem~\ref{t:SC_Oll_curv} is sharp in the sense that for any $\eps>0$,
there exists a stochastically incomplete graph with $\kap(r) \geq -(\log r)^{1+\eps}$. This can already
be seen from the case of birth-death chains for which stochastic completeness is equivalent to
$$\sum_{r=0}^\infty \frac{r+1}{b(r,r+1)}=\infty$$
by Theorem~\ref{t:sc_symmetry} above. For further details, see Theorem~4.11 in \cite{MW19}.

We note that the optimal curvature criterion for stochastic completeness in terms of curvature in the 
manifold setting gives the borderline for stochastic completeness around the curvature decay of order
$-r^2$, see \cite{Var83, Hsu89}.
One might be tempted to try and reconcile the difference between the manifold and graph setting
by using intrinsic metrics as was successfully carried out in the case of volume growth, however, for Ollivier Ricci
curvature this approach turns out to not work, see Example~4.13 in \cite{MW19}. 

We note that
this is not the only difference between the continuous and discrete settings when it comes to curvature.
As another example,
there exist infinite graphs which have uniformly positive Ollivier Ricci curvature, see Example~4.18 in \cite{MW19}.
This is known to be impossible in the manifold setting by the Bonnet--Myers theorem. Here,
anti-trees also prove to be a source of counterexamples as they provide examples of infinite 
graphs satisfying and $CD(K,\infty)$ for $K>0$ and well as having uniformly positive Ollivier Ricci curvature, 
see \cite{CLMP20} in \cite{KLW} for the Bakry--{\'E}mery and Ollivier
Ricci curvature of anti-trees. However, as soon as one either imposes upper bounds on the vertex degree or
lower bounds the measure, then a graph with uniformly positive lower curvature bounds 
must be finite, see \cite{LMP18} for the case of Bakry--{\'E}mery
and \cite{MW19} for the case of Ollivier Ricci.

\subsection*{Acknowledgements} I would like to thank J{\'o}zef Dodziuk
for suggesting such a fruitful area of study and for sustaining me
over the years.  I am also happy to acknowledge the inspiration and support offered by
Isaac Chavel and Leon Karp.
Furthermore, I would like to thank Alexander Grigor$'$yan for encouragement and support,
early in my career up until the present moment. And also many thanks to my coauthors who contributed 
to this story and whom I also consider to be good friends. In alphabetical order:
Sebastian Haeseler, Bobo Hua, Xueping Huang, Matthias Keller, Daniel Lenz, Jun Masamune, 
Florentin M{\"u}nch and Marcel Schmidt. 
Finally, I would like to thank Isaac Pesenson for the invitation to contribute
this article.

\end{document}